\documentclass[a4paper,11pt,reqno]{amsart}
\usepackage{epsfig,url,paralist}
\usepackage{amssymb}
\usepackage[normalem]{ulem}
\usepackage{fullpage,dynkin-diagrams}
\usepackage{amsmath}
\usepackage{tikz}
\usepackage{relsize}
\usepackage{fullpage}
\usepackage{hyperref,verbatim}
\usepackage[capitalize]{cleveref}
\usepackage{graphicx}
\usepackage{setspace,multirow}
\usepackage{enumitem,lineno}
\setlist{nolistsep}
\usepackage{lscape}
\usepackage{calrsfs}            % use a different \mathcal font
\usepackage{color}
\usepackage{colortbl}
\usepackage{xcolor}
\newtheorem{theo}{Theorem}

\numberwithin{equation}{section}

\newtheorem{theorem}{Theorem}[section]

\newtheorem{theorem*}{Main Result}
\newtheorem{lem}[theorem]{Lemma}
\newtheorem{coro}[theorem]{Corollary}
\newtheorem{cor*}[theorem*]{Corollary}
\newtheorem{obs}[theorem]{Observation}
\newtheorem{prop}[theorem]{Proposition}
\theoremstyle{definition}

\newtheorem{nota}[theorem]{Notation}
\newtheorem{remark}[theorem]{Remark}
\newtheorem{exm}[theorem]{Example}
\numberwithin{theorem}{section}
\usepackage{fancyhdr}
\def\<{\langle}
\def\>{\rangle}
\newcommand{\Res}{\mathsf{Res}}

\newcommand{\cS}{\mathcal{S}}
\newcommand{\pperp}{\perp\hspace{-0.15cm}\perp}
\newcommand{\GQ}{\mathsf{GQ}}
\newcommand{\per}{\text{\,\footnotesize$\overline{\land}$\,}}
\newcommand{\proj}{\mathsf{proj}}
\newcommand{\PGL}{\mathsf{PGL}}
\newcommand{\GL}{\mathsf{GL}}
\newcommand{\SL}{\mathsf{SL}}
\newcommand{\PSL}{\mathsf{PSL}}
\newcommand{\PGO}{\mathsf{PGO}}
\newcommand{\PGE}{\mathsf{PGE}}
\newcommand{\GE}{\mathsf{GE}}
\newcommand{\SE}{\mathsf{SE}}
\DeclareMathOperator{\Aut}{\mathsf{Aut}}
\DeclareMathOperator{\typ}{\mathsf{typ}}
\DeclareMathOperator{\cotyp}{\mathsf{cotyp}}
\newcommand{\PG}{\mathsf{PG}}
\newcommand{\Sc}{\mathsf{Sc}}
\newcommand{\id}{\mathsf{id}}
\newcommand{\K}{\mathbb{K}}
\newcommand{\F}{\mathbb{F}}
\newcommand{\Z}{\mathbb{Z}}
\renewcommand{\L}{\mathbb{L}}
\newcommand{\cC}{\mathcal{C}}

\newcommand{\cP}{\mathcal{P}}
\newcommand{\cE}{\mathcal{E}}
\newcommand{\cL}{\mathcal{L}}
\renewcommand{\O}{\mathbb{O}}

\DeclareMathOperator{\dist}{\mathsf{dist}}
\DeclareMathOperator{\diam}{\mathsf{diam}}
\DeclareMathOperator{\diag}{\mathsf{diag}}

\makeatletter
\@namedef{subjclassname@2020}{\textup{2020} Mathematics Subject Classification}
\makeatother
\begin{document}
\title{Groups of projectivities and Levi subgroups in spherical buildings of simply laced type}
\author{Sira Busch, Jeroen Schillewaert \and Hendrik Van Maldeghem}
\address{Sira Busch\\ Department of Mathematics, M\"unster University, Germany}
\email{s\_busc16@uni-muenster.de}
\address{Jeroen Schillewaert \\ Department of Mathematics, University of Auckland, \\New-Zealand}\email{j.schillewaert@auckland.ac.nz}
\address{Hendrik Van Maldeghem \\ Department of Mathematics, Computer Science and Statistics, Ghent University, Belgium} \email{Hendrik.VanMaldeghem@UGent.be}
\thanks{The first author is funded by the Claussen-Simon-Stiftung and by the Deutsche Forschungsgemeinschaft (DFG, German Research Foundation) under Germany's Excellence Strategy EXC 2044 --390685587, Mathematics M\"unster: Dynamics--Geometry--Structure. All authors were supported by the New Zealand Marsden Fund grant UOA-2122 of the second author. This work is part of the PhD project of the first author.} 
\maketitle
%\linenumbers
%\vspace{-1cm}

\begin{abstract}
We introduce the special and general projectivity groups attached to a simplex $F$ of a thick, irreducible, spherical building of simply laced type. If the residue of $F$ is irreducible, we determine the permutation group of both projectivity groups of $F$, acting on the residue of $F$ and show that the special projectivity group determines the precise action of the Levi subgroup of a parabolic subgroup on the corresponding residue. This reveals three special cases for the exceptional types $\mathsf{E_6,E_7,E_8}$. Furthermore, we establish a general diagrammatic rule to decide when exactly the special and general projectivity groups of $F$ coincide. 
\end{abstract}

%\tableofcontents

\section{Introduction}
The theory of buildings evolved during the search for analogues of exceptional simple Lie groups over arbitrary fields; traditionally people only worked over the fields $\mathbb{C}$ and $\mathbb{R}$. This was of interest, since working over arbitrary fields would allow the field to be finite and with that, one could find new families of finite simple groups. In 1955, Chevalley managed to construct these analogues and the groups he found are now known as \emph{Chevalley groups}. After Chevalley published his work, Jacques Tits developed the theory of buildings, attaching geometric structures to these groups (see \cite[page 335-335]{Abr-Bro:08}).

Chevalley groups defined over arbitrary fields are known to be \emph{groups of Lie type} (as in \cite{Car:72}). Groups of Lie type have BN-pairs and are hence associated to buildings (see \cite[page 108, Proposition 8.2.1]{Car:72}). They can be described as groups of automorphisms of spherical buildings (i.e.~buildings with finite Weyl groups, see \cite[\S 6.2.6]{Abr-Bro:08}). Chevalley groups are always simple, except in the cases $A_{1}(2)$, $A_{1}(3)$, $B_{2}(2)$, $G_{2}(2)$ (see \cite[page 172, Theorem 11.1.2]{Car:72}).

In this article we will focus on buildings of simply laced type and rank at least $3$. Such buildings automatically admit so-called \emph{root elations} (see \cite{Tits:77}). Then  we can define the Chevalley group attached to such a building $\Delta$ as the group of automorphisms generated by all root elations, which we will denote by $\Aut^{\dagger}(\Delta)$. This agrees with what is known as the the \emph{adjoint Chevalley group} (see \cite[page 198]{Car:72}), and is also called the \emph{little projective group of $\Delta$}. It is always simple in our cases, since we assume the rank to be at least $3$ (compare with \cite[Main Theorem]{Tits:64}).

Parabolic subgroups of Chevalley groups have attracted much attention in the literature. They can be written as semi-direct products of a \emph{unipotent subgroup} and a \emph{Levi subgroup} (see \cite[page 118]{Car:72}). So far, a lot of research focussed on the unipotent subgroups (see for example \cite{CurKanSei,GoLySo}). In this article we aim to shed some light on the Levi subgroups.

Let $\Sigma$ be an apartment of $\Delta$ and $C$ a chamber in $\Sigma$ that we will consider to be the fundamental chamber. Let $F$ be a face of $C$. A Levi subgroup of the parabolic subgroup $G_{F}$ of $\Aut^\dagger(\Delta)$ is a subgroup $L_{F}$, such that $G_{F}$ is the semi-direct product of $L_{F}$ and a unipotent subgroup. This matches with how it has been traditionally defined in the literature (see \cite[page 158, Definition 11.22]{Bor:69}). The parabolic subgroups opposite $G_{F}$ correspond bijectively to the Levi subgroups of $G_{F}$ (see \cite[page 199, Proposition 14.21]{Bor:69}). Hence a Levi subgroup fixes a simplex and a unique opposite simplex pointwise, and it acts as a group of automorphisms on the residue (also sometimes called the \emph{star}, for instance in \cite{Tits:74}) of each of these simplices. In the present paper, we determine the precise action of the Levi subgroup on that residue. To the best of our knowledge, this was not recorded before. However, it is known how it can be obtained by means of characters and co-characters of the torus of the corresponding Chevalley group, using the theory of roots
and co-roots in Chevalley groups. We explain and apply this approach in \cref{algapp}. It provides an algebraic answer to the problem of determining the precise action of the Levi subgroup on the corresponding residue of the building. 

However, approaching the problem in a geometric way via the theory of buildings gives rise to a new development of the theory of \emph{special and general projectivity groups}. The connection with the problem of the previous paragraph is given by our \cref{MR0}: it shows that the special projectivity group of $F$ coincides with the faithful permutation group induced by the stabiliser $\Aut^\dagger(\Delta)_F$ of $F$ in $\Aut^\dagger(\Delta)$ on the residue $\Res_\Delta(F)$ of $F$ in $\Delta$. This connects the special projectivity group of a simplex $F$ to the Levi subgroup of $F$. Since we determine all general and special groups of projectivities, this determines the precise action of the Levi subgroup of a parabolic subgroup on the corresponding residue geometrically.

As mentioned above, we also develop some basic and general theory about the projectivity groups. In projective geometry, the \emph{groups of projectivities}, or \emph{projectivity groups} play an important role in many proofs. For instance, projectivities between lines in a projective plane can be used to define non-degenerate conics (Steiner's approach) and prove properties of them. In \cite{Kna:88}, Knarr defined groups of projectivities and groups of even projectivities for generalised polygons and determined them in the finite case. This was further generalised to large infinite classes in \cite{Mal:98}, where the group of projectivities was called the \emph{general projectivity group} and the group of even projectivities the \emph{special projectivity group} related to a point or line. A generalisation of the definitions to all spherical buildings is obvious and natural questions are, for instance, 
\begin{itemize} 
\item when does the general projectivity group coincide with the special projectivity group, and 
\item can one  determine the various general and special  projectivity groups, particularly in the case where the residues are irreducible? 
\end{itemize}
In the present paper, we answer these questions for irreducible spherical buildings $\Delta$ with a simply laced diagram (see \cref{nsl} for the other cases). It will turn out that the special and general projectivity groups of residues of rank $1$ are always $\PGL_2(\K)$ in its natural permutation group action. This is Theorem D. For (irreducible) residues $R$ of rank at least $2$, in most cases we  generically obtain the maximal linear (algebraic or projective) group, including possible dualities, if opposition in the Coxeter diagram of the ambient building is trivial, and the one in the Coxeter diagram of $R$ is not trivial. There are only these four classes of exceptions:
\begin{compactenum}[$(i)$]
\item If $\Delta$ has type $\mathsf{D}_n$ and the type of $R$ contains the types $n-1$ and $n$ (hence $R$ is of type $\mathsf{D}_\ell$, for some $\ell<n$), then the projectivity groups are contained in $\PGO_{2\ell}(\K)$. Here, $\K$ is the underlying field. (Hence there are no proper similitudes in the projectivity groups.)
\item If $\Delta$ has type $\mathsf{E_6}$ and $R$ has type $\mathsf{A_5}$, then the special and general projectivity group consists of those members of $\PGL_6(\K)$ which correspond to matrices, for which the determinant is a third power in the field $\K$ of definition. 
 \item If $\Delta$ has type $\mathsf{E_7}$ and $R$ has type $\mathsf{A_5}$ containing the type $2$ (in Bourbaki labelling), then the special projectivity group consists of those members of $\PGL_6(\K)$, which correspond to matrices, for which the determinant is a square in the field $\K$ of definition. The general projectivity group extends this group with a duality; for instance a symplectic polarity, with corresponding matrix of square determinant.
 \item If $\Delta$ has type $\mathsf{E_7}$ and $R$ has type $\mathsf{D_6}$, then the special and general projectivity group are the simple group $\mathsf{P\Omega}_{12}(\K)$, extended with a class of diagonal automorphisms.
\end{compactenum} 
This is \cref{MR4}. A complete list in tabular form of all special and general projectivity groups acting on irreducible residues of buildings of type $\mathsf{E_6}, \mathsf{E_7}, \mathsf{E_8}$ and $\mathsf{D}_n$ (for $n \geq 4$) is included in Section 8. In our arguments, the so-called polar vertices of the diagram will play a crucial role, and our results will entail a new combinatorial characterisation of the polar type. Theorem B and C below show that these polar types are basically the only ones responsible for the special and general projectivity groups to coincide. 

We essentially provide two proofs for \cref{MR4}. The algebraic one, outlined in \cref{algapp}, and a geometric one in the remainder of \cref{allgroups}. However, the algebraic proof will make clear that the case of type $\mathsf{E_8}$ is, in a certain way, trivial---one always gets the full linear groups---so we skip this case in our geometric approach (however, it could be included and the geometric arguments will be contained in the first author's PhD thesis). The purpose of still providing the detailed geometric arguments for the other cases is the following: Firstly, we consider it interesting to see, where the exceptions mentioned above come from in a geometric way and how they emerge from the geometry of buildings. The geometry sometimes provides a different ``reason'' for a certain group to be the special projectivity group. Secondly, we also want to determine the general projectivity groups, which we can not obtain from the algebraic considerations. Thirdly, some geometric lemmata that we need, like for instance \cref{conn}, can be useful in other contexts, and lastly, it already prepares for handling the cases of non-simply laced types, where the question for non-split buildings is not so easily solved using the algebraic approach (think of buildings of so-called \emph{mixed type} that do not really admit a well-defined root system).  

The exceptions $(i)$ to $(iv)$ show that the questions stated above are not trivial and that the answer is rather peculiar, with exactly three special cases for the exceptional groups.

We now get down to definitions and statements of our Main Results. 
\section{Preliminaries and statement of the Main Results}
We will need some notions and notation related to spherical buildings, and of point-line geometries related to those. Excellent references for buildings are the books \cite{Abr-Bro:08} and \cite{Tits:74}, since it will be convenient to consider buildings as simplicial complexes. Standard references for the point-line approach to (spherical) buildings are \cite{Bue-Coh:13} and \cite{Shu:11}.

\subsection{Spherical buildings}\label{sphb} Let $\Delta$ be a spherical building. We will assume, as in \cite{Tits:74}, that $\Delta$ is a thick numbered simplicial chamber complex, and we will usually denote the type set by $I=\{1,2,\ldots,r\}$, where $r$ is the rank of $\Delta$, and the set of chambers by $\cC(\Delta)$. The type $\typ(F)$ of a simplex $F$ is the set of types of its vertices. A \emph{panel} is a simplex of size $r-1$. \emph{Adjacent} chambers are chambers intersecting in a panel. This defines in a natural way the \emph{chamber graph}. The \emph{(gallery) distance} $\delta(C,C')$ between two chambers $C$ and $C'$ is the distance in the \emph{chamber graph} of the vertices corresponding to $C$ and $C'$. 

One of the defining axioms of a spherical building is that every pair of simplices is contained in an \emph{apartment}, which is a thin simplicial chamber subcomplex isomorphic to a finite Coxeter complex $\Sigma(W,S)$ with associated Coxeter system $(W,S)$, where $W$ is a Coxeter group with respect to the generating set $S$ of involutions.
If $S=\{s_1,\ldots,s_r\}$, then let $P_i=\<s_1,\ldots,s_{i-1},s_{i+1},\ldots,s_r\>$, $i\in I$, be the maximal parabolic subgroups. The vertices of $\Sigma(W,S)$ of type $i\in I$ are the right cosets of $P_i$. The chambers are the sets of cosets of maximal parabolic subgroups containing a given member $w$ of $W$.   
 For each pair $(C,C')$ of adjacent chambers there exists exactly one \emph{folding}, that is, a type preserving idempotent morphism of $\Sigma(W,S)$ mapping $C'$ to $C$, and such that each chamber in the image has two chambers in its pre-image.   The image $\alpha$ of a folding is called a \emph{root}. The root associated to the \emph{opposite folding}, namely, the folding mapping $C$ to $C'$ is called the \emph{opposite} root, and is denoted by $-\alpha$. The intersection $\alpha\cap(-\alpha)$, called a \emph{wall}, is denoted by $\partial\alpha$ (and hence also by $\partial(-\alpha)$), and is also referred to as the \emph{boundary of $\alpha$}.   Every root contains a unique simplex that is fixed under  each automorphism of $\Sigma(W,S)$ preserving $\alpha$ (and not necessarily type preserving). This simplex is called the \emph{centre} of the root. If $\Sigma(W,S)$, or equivalently, $\Delta$,  is irreducible (see below), the type of such simplex is called a \emph{polar type} of $\Delta$. In the reducible case, the polar types of the connected components will be called \emph{polar types} of the building. 

For each vertex $v$ of $\Sigma(W,S)$, there exists a unique other vertex $v'$ of $\Sigma(W,S)$ with the property that every wall containing $v$ also contains $v'$ (and then automatically every wall containing $v'$ contains $v$); then $v$ and $v'$ are called \emph{opposite} vertices. \emph{Opposite} simplices of $\Sigma(W,S)$ are two simplices $A,B$ with the property that the vertex opposite to any vertex in $A$ is contained in $B$, and vice versa.  We denote $A\equiv B$. Opposition defines a permutation, also denoted by $\equiv$,  of order at most $2$ on the type set $I$. A subset $J\subseteq I$ is called \emph{self-opposite}, if $J^\equiv=J$.  The permutation $\equiv$, acting on $I$, induces an automorphism of the corresponding Coxeter diagram. Recall that the vertices of the Coxeter diagram correspond to the types, that is, the elements of $I$, and two vertices $i$ and $j$ are connected by an edge of weight $m_{ij}-2$, where  
$m_{ij}$ is the order of $s_is_j$ in $W$. Throughout, we use the Bourbaki labelling of connected spherical Coxeter diagrams \cite{Bou:68}. The Coxeter diagram, and by extension the chamber complex $\Sigma(W,S)$ and the building $\Delta$, are called \emph{simply laced}, if $m_{ij}\in\{2,3\}$, for all $i,j\in\{1,2,\ldots,r\}$, $i\neq j$. The building $\Delta$ is irreducible, if the Coxeter diagram is connected. The polar type in the simply laced and irreducible case is unique. It is the set of nodes to which the additional generator is joined, when constructing the affine diagram.
Hence, it is $\{1,r\}$ in case $\mathsf{A}_r$, it is $2$ in case of $\mathsf{D}_r$, and $2,1,8$ for $\mathsf{E_6,E_7,E_8}$, respectively. 

Opposite simplices in $\Delta$ are simplices that are opposite in some apartment, and then the building axioms guarantee that they are opposite in every apartment in which they are both contained.

We say that a vertex $v$ and a simplex $F$ are \emph{joinable}, if $v\notin F$ and $F\cup\{v\}$ is a simplex; notation $v\sim F$. (Note that we denote simplices with capital letters, such as $F$, since the letter $S$ already has a meaning. The letter $F$ stands for ``flag'', which is a synonym of simplex in the language of geometries.) The simplicial complex induced on the vertices joinable to a given simplex $F$ of a building $\Delta$ forms a building called the \emph{residue of $F$ in $\Delta$} and is denoted by $\Res_\Delta(F)$. It is well known that the Coxeter diagram of that residue is obtained from the Coxeter diagram of the building by deleting the vertices with type in $\typ(F)$. The opposition relation in $\Res_\Delta(F)$ will be denoted by $\equiv_F$ (also on the types), and two simplices of $\Res_\Delta(F)$, opposite in $\Res_\Delta(F)$, will occasionally be called \emph{locally opposite at $F$}. The cotype $\cotyp(F)$ of a simplex $F$ is $I\setminus\typ(F)$, and the \emph{type} of the residue $\Res_\Delta(F)$ is the cotype of $F$. 

Now let $F$ and $F'$ be two opposite simplices. Let $C\in\cC(\Delta)$ be such that $F\subseteq C$. Then there exists a unique chamber $C'\supseteq F'$ at minimal gallery distance from $C$. The chamber $C'$ is called the \emph{projection of $C$ from $F$ onto $F'$} and denoted $\proj_{F'}^F(C)$.  That mapping is a bijection from the set of chambers through $F$ to the set of chambers through $F'$ and preserves adjacency in both directions. It follows that it defines a unique isomorphism from $\Res(F)$ to $\Res(F')$, which we denote by $\proj_{F'}^F$ (as it is a special case of the projection operator, see 3.19 of \cite{Tits:74}), see also Theorem 3.28 of \cite{Tits:74}.  When the context makes $F$ clear, we sometimes remove the $F$ from the notation for clarity and simply write $\proj_{F'}$.  This projection has the following property.

\begin{prop}[Proposition~3.29 of \cite{Tits:74}] \label{Tits}
Let $F$ and $F'$ be opposite simplices of a spherical building $\Delta$. Let $v$ be a vertex of $\Delta$ such that $v\sim F$, and set $i:=\typ(v)\in I$. Then the type $i'$ of the vertex $\proj^{F}_{F'}(v)$ is the opposite in $\Res(F')$ of the opposite type of $i$ in $\Delta$, that is, $i'=(i^\equiv)^{\equiv_{F'}}$.  Also, vertices $v\sim F$ and $v'\sim F'$ are opposite in $\Delta$ if, and only if, $v'\equiv_{F'}\proj^{F}_{F'}(v)$.   
\end{prop}

Now let $\Delta$ be irredicible and of rank at least $2$. Let $\alpha$ be a root of $\Delta$. Let $U_\alpha$ be the group of automorphisms of $\Delta$ pointwise fixing every chamber that has a panel in $\alpha\setminus\partial\alpha$. The elements of $U_\alpha$ are called \emph{(root) elations} and $U_\alpha$ itself is called a \emph{root group}. An element of $U_\alpha$ is called a \emph{central elation} if it belongs to $U_\beta$ for each root $\beta$ having the same centre as $\alpha$. If $U_\alpha$ acts transitively on the the set of apartments containing $\alpha$, then we say that $\alpha$ is \emph{Moufang}. If every root is Moufang, then we say that $\Delta$ is Moufang. The automorphism group of $\Delta$ is denoted by $\Aut\:\Delta$ and, if $\Delta$ is Moufang, then the subgroup generated by the root elations is denoted by $\Aut^\dagger\Delta$ and called the \emph{little projective group of $\Delta$}. It is also generated by all central elations (in the simply laced case, all elations are central). Also, we denote by $\Aut^\circ(\Delta)$ the subgroup of type-preserving automorphisms of $\Delta$. In the literature, this is also sometimes denoted as $\mathrm{Spe}(\Delta)$.  Finally, each irreducible spherical building $\Delta$ of rank $r\geq 3$ is Moufang. 

\subsection{Groups of projectivity} Let $\Delta$ be a spherical building and $F$, $F'$ two simplices which are opposite, and which are not chambers. Then we call the isomorphism $\proj^F_{F'}$ a \emph{perspectivity (between residues)} and denote $F\per F'$. If $F_0,F_1,\ldots,F_{\ell}$ is a sequence of consecutively opposite simplices, then the isomorphism $\Res(F_0)\to\Res(F_\ell)$ given by $\proj^{F_{\ell-1}}_{F_\ell}\circ\cdots\proj^{F_1}_{F_2}\circ\proj^{F_0}_{F_1}$ is called a \emph{projectivity (of length $\ell$)}. If $\ell$ is even, it is called an \emph{even} projectivity, and if $F_0=F_\ell$, it is called a \emph{self-projectivity}. The set of all self-projectivities of a simplex $F$ is a group called the \emph{general projectivity group of $F$} and denoted $\Pi(F)$. Likewise, the set of all even self-projectivities of a simplex $F$ is a group called the \emph{special projectivity group of $F$} and denoted $\Pi^+(F)$.  Note that $\Pi(F)=\Pi^+(F)$ as soon as $(\typ(F))^\equiv\neq \typ(F)$.

Let $\Pi(F)$ be the general projectivity group of the simplex $F$ of a spherical building $\Delta$, with $F$ not a chamber. Then, as an abstract permutation group, $\Pi(F)$ only depends on the type of $F$. Likewise, the special projectivity group  $\Pi^+(F)$ only depends on the type of $F$. We have the natural inclusion $\Pi^+(F)\unlhd\Pi(F)$ and $[\Pi(F):\Pi^+(F)]\leq 2$. We denote the number $[\Pi(F):\Pi^+(F)]$ by $n(J)$, where the type of $F$ is $J$. We trivially have $n(J)=n(J^\equiv)$, because it is $1$ if $J^\equiv\neq J$. 

In the case that $\Delta$ has rank 2, that is, $\Delta$ is the building of a generalised polygon, $F$ is necessarily a single vertex and can be thought of as either a point (type $1$) or a line (type $2$) of the generalised polygon. Knarr \cite{Kna:88} shows that, if $\Delta$ is Moufang, then for every point or line $x$ of $\Delta$, the group $\Pi^+(x)$ coincides with the stabiliser of $x$ in the little projective group of $\Delta$, that is, the group generated by the root groups. We generalise this to arbitrary simplices in arbitrary Moufang spherical buildings of simply laced type. This is our first main result, \cref{MR0}. The strategy of the proof is the same as for the rank $2$ case. However, the proof requires that the unipotent radical of a parabolic subgroup in a Moufang spherical building pointwise stabilises the corresponding residue, and acts transitively on the simplices opposite the given residue. This follows from the Levi decomposition of parabolic subgroups in Chevalley groups. We provide a brief introduction. 

\subsection{The Levi decomposition in Chevalley groups} 
Let $\Delta$ be a building and $F$ a simplex of type $J$. Suppose $\Delta$ is of irreducible simply laced type and has rank at least $3$. Then, by the classification in \cite{Tits:74}, $\Delta$ is Moufang. Its little projective group $\Aut^\dagger(\Delta)$ is either a Chevalley group, or, in case $\Delta$ corresponds to a projective space of dimension $d$ defined over a non-commutative skew field $\L$,  it is  $\PSL_{d+1}(\L)$ (in its natural action). The stabiliser $P_{F}$ of $F$ is called a \emph{parabolic subgroup} and, if $\Aut^\dagger(\Delta)$ is a Chevalley group, admits a so-called \emph{Levi decomposition} $P_{F}=U_{F}L_{F}$, see Section~8.5 of \cite{Car:72}, where $U_{F}$ is the so-called \emph{unipotent radical} of $P_{F}$ and $L_{F}$ is called a \emph{Levi subgroup}. 

We provide an explicit description of $P_{F}, U_{F}$ and $L_{F}$ for $\PSL_{d+1}(\L)$ in the case that we will need most in the present paper, namely when $\Res_\Delta(F)$ is irreducible.  In that case one chooses the basis in such a way that each subspace of $F$ of dimension $i$ is generated by the first $i+1$ base points. Also, $F$ consists of $i$-dimensional subspaces with $0\leq i\leq d_1-1$ and $d-d_3\leq i\leq d-1$, where $|F|=d_1+d_3$. Set $d_2:=d+1-d_1-d_3$. Note that $J=\{1,\ldots,d_1,d-d_3+1,\ldots, d\}$. Then a generic element of $P_F$ looks like
\[\begin{pmatrix} T_{d_1} & M_{d_1\times d_2} & M_{d_1\times d_3} \\ O_{d_2\times d_1} & M_{d_2\times d_2} & M_{d_2\times d_3} \\ O_{d_3\times d_1} & O_{d_3\times d_2} & T_{d_3} 
\end{pmatrix},\]
where $T_{d_i}$, $i=1,3$, is an arbitrary invertible upper triangular matrix over $F$ (needless to say that $T_{d_1}$ and $T_{d_3}$ are independent of each other; even if $d_1=d_3$ they are considered different), $M_{d_i\times d_j}$ is an arbitrary $d_i\times d_j$ matrix, $i\in\{1,2\}$ and $j\in\{2,3\}$ (with similar remark as for the $T_{d_i}$), and $O_{d_i\times d_j}$ is the $d_i\times d_j$ zero matrix, $i\in\{2,3\}$, $j\in\{1,2\}$.  Also, the Dieudonn\'e determinant of the whole matrix must be $1$. With similar notation, and on top with $U_{d_i}$, $i\in\{1,3\}$, an arbitrary unipotent upper triangular $d_i\times d_i$ matrix, $D_{d_i}$, $i\in\{1,3\}$ an arbitrary invertible diagonal $d_i\times d_i$ matrix and $I_{d_2}$ the $d_2\times d_2$ identity matrix, generic elements of $U_F$ and $L_F$ look like (blanks replace zero matrices)
\[\begin{pmatrix} U_{d_1} & M_{d_1\times d_2} & M_{d_1\times d_3} \\ %O_{d_2\times d_1} 
& I_{d_2} & M_{d_2\times d_3} \\ %O_{d_3\times d_1} & O_{d_3\times d_2} 
&& U_{d_3} 
\end{pmatrix} \mbox{ and }\begin{pmatrix} D_{d_1} & %M_{d_1\times d_2} & M_{d_1\times d_3}
 \\ %O_{d_2\times d_1} 
 & 
 M_{d_2\times d_2} & %M_{d_2\times d_3} 
 \\ %O_{d_3\times d_1} & O_{d_3\times d_2} 
 && D_{d_3} 
\end{pmatrix},\] respectively (with again the requirement that the Dieudonn\'e determinant of the second matrix is equal to $1$). One indeed checks that $P_{F}=U_{F} L_{F}$ and $U_{F}\cap L_{F}=\{I_{d+1}\}$. Also, the following lemma is easily checked in this case. For Chevalley groups (the case of relevance for the current paper), the lemma follows from the Levi decomposition of parabolic subgroups (see Section~8.5 of \cite{Car:72}). We state it in its most general form, as given and proved in  
\cite[Proposition~24.21]{Muh-Ped-Wei:15}
 
\begin{lem}\label{L1}
Let $\Delta$ be a spherical Moufang building and let $F$ be a simplex of $\Delta$ of type $J$. Let $P_{F}$ be the stabiliser of $F$ in $\Aut^\dagger(\Delta)$. Then the unipotent radical $U_F\leq P_{F}$ acts sharply transitively on the set $F^\equiv$ of simplices opposite $F$, and pointwise fixes $\Res_\Delta(F)$.  
\end{lem}

%\begin{proof}
%From the definition of $U_F$ in Section~8.5 of \cite{Car:72}, it is readily deduced that $U_F$ pointwise fixes $\Res_\Delta(F)$, because it is generated by root groups, whose corresponding roots contain $F$, but not in their boundary.  Furthermore, a Levi subgroup $L_F$ is just the stabiliser in $G$ of $F$ and an opposite simplex $F'$ as described in the introduction. Since $P_F$ acts transitively on $F^\equiv$ (by the BN-pair property of Chevalley groups), we find that $U_F$ acts transitively on $F^\equiv$. Since $U_F\cap L_F$ is just the identity (see Theorem~8.5.2 of \cite{Car:72}), the lemma follows. 
%\end{proof}

We will be interested in the faithful permutation group induced by $L_F$ on $\Res_\Delta(F)$. 

\subsection{Main results}

\begin{theo}\label{MR0}
Let $F$ be a simplex of a Moufang spherical building $\Delta$. Let $\Aut^\dagger(\Delta)$ be the automorphism group of $\Delta$ generated by the root groups. Then $\Pi^+(F)$ is permutation equivalent to the faithful permutation group induced by the stabiliser $\Aut^\dagger(\Delta)_F$ of $F$ in $\Aut^\dagger(\Delta)$ on the residue $\Res_\Delta(F)$ of $F$ in $\Delta$. 
\end{theo}

Going back to the case where $\Delta$ is a Moufang building of rank $2$, the results in Chapter 8 of \cite{Mal:98} show that $n(\{1\})=n(\{2\})=1$, as soon as $\Delta$ is a so-called ``Pappian polygon'' (for a definition of the latter, see Section 3.5 of \cite{Mal:98}). In any case, we always have $1\in\{n(\{1\}),n(\{2\})\}$ due to Lemma~8.4.6 of \cite{Mal:98}. One of the goals of the present paper is to generalise this to all spherical buildings.  This will be achieved by proving a general sufficient condition in $J$ for $n(J)$ being equal to $1$. To state this, we say that the type $J$ of a simplex is \emph{polar closed}, if we can order the elements of a partition of $J$ into singletons and pairs, say $J_1,\ldots,J_k$, such that, for each $\ell\in\{1,\ldots,k\}$, the type $J_\ell$ is a polar type in the residue of $J_1\cup\cdots\cup J_{\ell-1}$. We then have:

\begin{theo} \label{MR1}
Let $\Delta$ be a spherical building with type set $I$. If either $J\neq J^\equiv$ or $J\subseteq I$ is polar closed, then $n(J)=1$.  
\end{theo}

To see a partial converse of this statement, we restrict to the simply laced case (see also \cref{nsl}). 

\begin{theo}\label{MR2}
Let $\Delta$ be an irreducible spherical building of simply laced type with type set $I$. If $J\subseteq I$, $J^\equiv=J$ and $I\setminus J$ has at least one connected component $K$ of size at least $2$, such that $I\setminus K$ is not polar closed, then $n(J)=2$. 
\end{theo}
Note that, if $J$ is polar closed, then for each connected component $K$ of $I\setminus J$ the type set $I\setminus K$ is polar closed. 

This implies the following combinatorial characterisation of the polar type in connected simply laced spherical diagrams. For $K\subseteq I$ we denote by $\overline K$ the union of all connected components of $K$ of size at least $2$.

\begin{cor*}\label{MCor}
The polar type of a connected simply laced spherical diagram $D_I$ over the type set $I$ is the unique smallest subset $J\subseteq I$ with the property that opposition in $D_{\overline{I\setminus J}}$ coincides with opposition in $D_I$.
\end{cor*}

\cref{MCor} does not hold in the non-simply laced case (since opposition does not determine the direction of the arrow in the Dynkin diagram). Indeed, for types $\mathsf{B}_n$, $\mathsf{C}_n$ and $\mathsf{F_4}$, there are each time two single types satisfying the given condition, reflecting the fact that, in characteristic $2$, there are really two choices. 

Finally, we consider the case left out in \cref{MR2} above, where $I\setminus J$ has only connected components of rank $1$. We reduce the action of $\Pi^+(F)$ on each panel to a case where $|I\setminus J|=1$ and show: 

\begin{theo}\label{MR3} Let $\Delta$ be an irreducible spherical building of simply laced type with type set $I$. Let $J\subseteq I$ with $|I\setminus J|=1$, and let $P$ be a panel of type $J$. Then $\Pi^+(P)$ is permutation equivalent to the natural action of $\PGL_2(\K)$ on the projective line $\PG(1,\K)$, and equals $\Pi(P)$. 
\end{theo}

In view of \cref{MR3}, one could expect that the general and special projectivity groups of simplices, whose residue is isomorphic to $\mathsf{A}_r(\K)$, are isomorphic to $\PGL_{r+1}(\K)$. This is indeed in most cases true, but not always. If it is not true, then necessarily the residue in question is not contained in a larger residue of type $\mathsf{A}_{r+1}$. Our last main result determines the exact permutation representations of the special and general projectivity groups on the corresponding residues of the building.

\begin{theo}\label{MR4} 
Let $\Delta$ be an irreducible spherical building of simply laced type with type set $I$. Let $I\neq J\subseteq I$ with $I\neq I\setminus J$ connected and let $F$ be a simplex of type $J$. Then $\Pi^+(F)$ and $\Pi(F)$ are 
\begin{compactenum}[$(i)$]
\item isomorphic to $\PGL_n(\L)$ in its natural action, if $\Delta$ has type $\mathsf{A}_r$, $r\geq 2$, it is defined over the skew field $\L$, and $|I\setminus J|=n-1$;
\item as in \emph{\cref{Drtabel}} and \emph{\cref{Etabel}} for $\typ(\Delta)\in\{\mathsf{D}_r,\mathsf{E}_m\mid r\geq 4, m=6,7,8\}$.
\end{compactenum}
\end{theo}

The notation used in Tables~\ref{Drtabel} and~\ref{Etabel} is explained in \cref{allgroups}, where \cref{MR4} is proved. 

\subsection{Lie incidence geometries}
Some arguments --- in particular those in \cref{allgroups} --- will be more efficiently carried out in a specific point-line geometry related to the spherical building in question.  We provide a brief introduction here. More details can be found in textbooks like \cite{Bue-Coh:13} and \cite{Shu:11}.

\subsubsection{Point-line geometries, projective spaces, polar spaces and parapolar spaces} Recall that a \emph{point-line} geometry $\Gamma=(X,\cL)$ consists of a set $X$, whose elements are called \emph{points}, and a subset $\cL$ of the full set of subsets of $X$, whose members are called \emph{lines} (hence we disregard geometries with so-called repeating lines). The notion of \emph{collinear points} will be used frequently. We denote collinearity of two points $x$ and $y$ with $x\perp y$, and $x^\perp$ has the usual meaning of the set of points collinear to $x$ (including $x$, if there exists a line containing $x$).  A \emph{(proper) subspace} is a (proper) subset of the point set intersecting each line in either 0,1 or all of the points of the line. A \emph{(proper) hyperplane} is a (proper) subspace intersecting each line non-trivially. The \emph{point graph} of $\Gamma$ is the graph with vertices the points, adjacent when collinear. A subspace is \emph{convex}, if its induced subgraph in the point graph is convex (all vertices on paths of minimal length between two vertices of the subspace are contained  in the subspace). We will frequently regard a subspace as a subgeometry in the obvious way. A subspace is called \emph{singular}, if every pair of points in it is collinear. In our cases, singular subspaces will always be projective spaces. Lines and planes are short for $1$- and $2$-dimensional projective (sub)spaces, respectively.   

The \emph{distance} between points is the distance in the point graph and the \emph{diameter} of the geometry is the diameter of the point graph. 

We usually require that $\Gamma$ is \emph{thick}, that is, each line contains at least three points.

For example, the $1$-spaces of any vector space $V$ of dimension at least $3$ over some skew field $\L$, form the point set of a thick geometry $\PG(V)$, a generic line of which consists of all the $1$-spaces contained in a given $2$-space. This geometry is a projective space. The hyperplanes correspond to the codimension $1$ subspaces of $V$.  

A \emph{polar space} is a thick point-line geometry, such that for each point $x$, the set $x^\perp$ is a hyperplane (which we require to be distinct from the whole point set). 

A pair of points of a point-line geometry $\Gamma$ is called \emph{special}, if they are not collinear and there is a unique point of $\Gamma$ collinear to both. Then $\Gamma$ is called a \emph{parapolar space}, if every non-special pair of points at distance at most $2$ is contained in a convex subspace isomorphic to a polar space. Such convex subspaces are called \emph{symplecta}, or \emph{symps} for short. A pair $p,q$ of non-collinear points of a symp is called \emph{symplectic}; in symbols $p \pperp q$. 

Given an irreducible spherical building $\Delta$ of rank $r$ at least $2$ of type $\mathsf{X}_r$ over the type set $I$, let $J\subseteq I$ and define $X$ as the set of all simplices of $\Delta$ of type $J$. The set $\cL$ of lines consists of the sets of simplices of type $J$ completing a given panel, whose type does not contain $J$, to a chamber. The geometry $(X,\cL)$ is usually referred to as the \emph{Lie incidence geometry of type $\mathsf{X}_{r,J}$} (where we replace $J$ by its unique element, if $|J|=1$). The main observation here (see the above references), usually referred to as \emph{Cooperstein's theory of symplecta} \cite{Coo:76,Coo:77}, is that $(X,\cL)$ is either a projective space, a polar space, or a parapolar space. 

In the present paper, we will only use projective spaces over arbitrary skew fields (they are related to buildings of type $\mathsf{A}_r$), polar spaces (that are related to buildings of type $\mathsf{D}_r$),  some specific parapolar spaces that are related to buildings of types $\mathsf{E_{6}}$ and $\mathsf{E_{7}}$ over a field $\K$, and, at the end of this section, the thin exceptional long root subgroup geometries, which can be seen as parapolar spaces with special pairs of points and diameter 3. Polar spaces related to buildings of type $\mathsf{D}_r$ will usually be called \emph{polar spaces of type $\mathsf{D}_r$}, or \emph{hyperbolic} polar spaces, since in rank $r\geq 4$, they are in one-to-one correspondence to hyperbolic quadrics in projective spaces. Recall that a \emph{hyperbolic quadric} is the projective null set of a quadratic form of maximal Witt index in a vector space $V$ of even dimension. The standard form (using coordinates $x_{-r},\ldots,x_{-1},x_1,\ldots, x_r$) is given by \[x_{-r}x_r+x_{-r+1}x_{r-1}+\cdots x_{-2}x_2+x_{-1}x_1.\] 
The automorphisms of $\Delta$ induced by elements of $\PGL(V)$ will be called \emph{linear}. They conform to the elements of the corresponding (maximal) linear algebraic group. Note that hyperbolic quadrics contain two natural classes of maximal singular subspaces, characterised by the fact that members of distinct classes intersect in subspaces of odd codimension (the \emph{codimension} is the vector dimension of a complementary subspace).

Concerning types $\mathsf{E}_r$, $r=6,7,8$, we list some basic properties of the Lie incidence geometries of types $\mathsf{E_{6,1}}$ and $\mathsf{E_{7,7}}$,  that we will make use of, in \cref{propE6E7}. 

We end this section with the following lemma, whose proof makes use of thin exceptional long root subgroup geometries.

\begin{lem}\label{centreroot}
Let $\alpha$ be a root of an irreducible spherical Coxeter complex $\Sigma(W,S)$. Let $F$ be the centre of $\alpha$. Let $v$ be a vertex joinable to a vertex $u$ that is joinable to $F$. Then $v$ lies in $\alpha$.
\end{lem}

\begin{proof}
For the classical types $\mathsf{A}_n,\mathsf{B}_n$ and $\mathsf{D}_n$, this is easily verified. Indeed, for type $\mathsf{A}_n$, viewing $\Sigma(W,S)$ as the point-line geometry of a projective space with two points per line, $F$ is an incident point-hyperplane pair. Then $v$ is either incident with the point, or with the hyperplane and the result follows. For types $\mathsf{B_n}$ and $\mathsf{D}_n$, we view $\Sigma(W,S)$ as a polar space with two points per line. There are two possibilities: either $F$ is a point, and then $v$ is a subspace collinear to $F$ and the assertion again follows, or $F$ is a line, and then $v$ is either a subspace collinear to $F$, or a subspace containing a point of $F$. In both cases the assertion follows. 

Now let $\Sigma(W,S)$ have exceptional type. The rank 2 case is easy to check, so we may assume the type is $\mathsf{E}_i$, $i=6,7,8$, $\mathsf{F_4}$ or $\mathsf{H}_i$, $i=3,4$. The latter cases are not essential for us, as they do not correspond to thick buildings and so we leave this to the reader. In the other cases, we use the representation of $\Sigma(W,S)$ as \emph{thin long root subgroup geometry} $\Gamma$, that is, the geometry where points are the long roots of the corresponding root system, and lines (edges if one considers this geometry as a graph) given by pairs of roots making an angle of sixty degrees, see also \cite{Bue-Coh:13}. Such geometries are depicted for all exceptional types in \cite{Mal-Vic:19}. The advantage of this description is that $F$ is a point of this geometry (for type $\mathsf{F_4}$ one has also to consider the same construction with short roots, which gives an isomorphic geometry). Also, $\alpha$ is induced by all points corresponding to roots making an angle of at most $90$ degrees with the root corresponding to $F$, or, in other words, points collinear or symplectic to $F$. Now, every vertex of $\Sigma(W,S)$ corresponds to either a singular subspace of   $\Gamma$, to a symplecton of $\Gamma$, or a convex subspace isomorphic to a Lie incidence geometry containing no special pairs and having diameter 2. It is now clear that, if not both $u$ and $v$ are convex subspaces distinct from symplecta, then the assertion follows (as $v$ is only incident with points collinear or symplectic to $F$). The only case where both $u$ and $v$ are convex subspaces occurs for type $\mathsf{E_6}$, where, up to duality, $u$ is a vertex of type $1$ and $v$ of type $6$ (the corresponding convex subspaces are geometries of type $\mathsf{D_{5,5}}$). Since $u$ and $v$ are incident, they share a symp, and we may assume that symp is opposite $F$ in the convex subspace $u$. Then one verifies that $v$ contains four points collinear to $F$, eight points symplectic to $F$ and four points special to $F$,   This is a symmetric configuration with respect to $F$ and its opposite point, hence $u$ lies in $\partial\alpha$. 

The lemma is proved.
\end{proof}

\subsection{A connectivity theorem}
We will also need the connectivity of the subgeometry of a Lie incidence geometry of type $\mathsf{E_{6,2}}$, $\mathsf{E_{7,1}}$ or $\mathsf{E_{7,3}}$ induced by the points opposite two given points of the geometry. In an earlier version of the current paper, we proved this inside certain relevant parapolar spaces. The referee made us aware of a more general approach valid for all Lie incidence geometries defined in spherical buildings of simply laced type.  We present this approach here. Hence, the aim of this subsection is to prove the following proposition, which more generally also holds for twin buildings of simply laced type, but since we did not define these, we do not insist. 

\begin{prop}\label{connopp}
Let $\Delta$ be an irreducible spherical building of simply laced type such that each panel is contained in at least four chambers. Let $C,C'$ be two arbitrary chambers of $\Delta$. Then the subgraph $\Gamma^{C,C'}$ of the chamber graph $\Gamma$ induced on the set of chambers opposite both $C$ and $C'$ is connected.  
\end{prop}

The proof we present is rather similar to the proof of \cite[Theorem~5.1]{Muh-Ron:95}, which is essentially the case $C=C'$ of \cref{connopp}. So, we first verify the latter for rank 2 residues.

\begin{lem}\label{connopp2}
Let $\Delta$ be a spherical building of type $\mathsf{A_1\times A_1}$ or $\mathsf{A_2}$ such that each panel is contained in at least four chambers. Let $C,C'$ be two arbitrary chambers of $\Delta$. Then the subgraph $\Gamma^{C,C'}$ of the chamber graph $\Gamma$ induced on the set of chambers opposite both $C$ and $C'$ is connected.  
\end{lem}

\begin{proof}For $C=C'$, this is straightforward, hence assume that $C\neq C'$.

For type $\mathsf{A_1\times A_1}$, the chamber graph is a grid, and hence the said subgraph is a subgrid, which is always connected.

Suppose now $\Delta$ stems from a projective plane with at least $4$ points per line. First suppose each line has exactly $4$ points. Then one verifies easily that, if $C$ and $C'$ are adjacent, then $\Gamma^{C,C'}$  is the chamber graph of a $3\times3$ grid. 
If $C$ and $C'$ have distance 2 in the chamber graph, then one verifies that $\Gamma^{C,C'}$ is a cycle of length $18$, hence connected. Finally, if $C$ and $C'$ are opposite, then $\Gamma^{C,C'}$ consists of four $3$-cliques $\{a_{-2},a_{-1},a_0\}$, $\{a_0,a_1,a_2\}$, $\{b_{-2},b_{-1},b_0\}$ and $\{b_0,b_1,b_2\}$ and edges $\{c_i,b_i\}$ and $\{c_i,a_i\}$, $i\in\{-2,-1,1,2\}$.

So, we may suppose that each line has at least five points. Let $D,D'$ be two chambers both opposite both $C$ and $C'$. We observe that, if $D$ and $D'$ have distance $2$, then the unique chamber $E$ adjacent to both $D$ and $D'$ is also opposite both $C$ and $C'$. (This is most easily seen considering chambers as flags of the corresponding projective plane. Indeed, then $E$ consists of a point of either $D$ and $D'$, and a line of either $D'$ or $D$, respectively. Hence, since $C$ is opposite both $D$ and $D'$, its elements are not incident and do not coincide with any of the elements of $D$ and $D'$, and hence neither with the elements of $E$.) Hence we may assume that $D$ and $D'$ are opposite. Let $S$ be the set of chambers sharing their point with $D$ and let $S'$ be the set of chambers sharing their line with $D'$; note that $S$ and $S'$ are opposite panels. For each chamber $B\in S$ there exists a unique chamber $B'\in S'$ at distance $2$ in the chamber graph. Let $C_B$ be the unique chamber adjacent to both $B$ and $B'$. Then we observe (with a similar proof as our first observation above) that any chamber opposite both $D$ and $D'$ is not opposite at most two members of $U=\{C_B\mid B\in S\}$. In particular, $C$ is not opposite at most two members of   $U$, and $C'$ is not opposite at most two members of $U$. Consequently, there is at least one member of $U$ opposite both $C$ and $C'$ and our first observatin now implies that there is a path from $D$ to $D'$ inside $\Gamma^{C,C'}$. 
\end{proof}

We can now prove \cref{connopp}. Let $D$ and $D'$ be two chambers belonging to $\Gamma^{C,C'}$.  Let $\gamma$ be a path in $\Gamma^{C,C}$ connecting $D$ with $D'$ such that, among all such paths, the minimal distance $d$ from $C'$ to any member of $\Gamma$ is the highest, and the number $n$ of chambers attaining that minimal distance is smallest.    We claim that $d$ is the diameter of $\Gamma$, which shows that $\gamma$ is inside $\Gamma^{C,C'}$. Indeed, suppose there are elements of $\gamma$ not opposite $C'$ and let $D_1$ be the first chamber of $\gamma$ having distance $d$ to $C'$. Let $D_0$ be the element of $\gamma$ preceding $D_1$ and $D_2$ the one following $D_1$. Then our assumptions imply that $F=D_0\cap D_1\cap D_2$ has corank~2, and so $R:=\Res_\Delta(F)$ is a rank 2 building corresponding to either a generalised digon or a projective plane.  Since all elements of $\gamma$ are oposite $C$, the projection $C_R$ of $C$ onto $R$ is opposite all of $D_0,D_1,D_2$. Let $C_R'$ be the projection of $C'$ onto $R$. Let $n_i$, $i=0,1,2$, be the distance in the chamber graph of $R$ from $C_R'$ to $D_i$. Since projections of chambers onto panels are unique, it is straightforward to find paths $\gamma_i$, $i=1,2$, of length $2-n_i$ (if $R$ corresponds to a generalised digon) of $3-n_i$ (if $R$ corresponds to a projective plane) connecting $D_i$ with a chamber of $R$ opposite $C_R'$. Then, using \cref{connopp2}, we obtain a path in $R$ connecting $D_0$ with $D_2$ having one chamber less at distance $d_1$ from $C_R'$, and all other chambers have distance at least $d_1+1$. Replacing $D_0,D_1,D_2$ by this path, we obtain a path in $\Gamma^{C,C}$ with either higher minimal distance $n$, or less chambers at that minimal distance, a contradiction. The proposition is proved. \qed

\section{General observations and proof of \cref{MR0}}

We start this section with a simple, though important observation, used in both \cite{Kna:88} and Chapter 8 of \cite{Mal:98}, but not explicitly stated in either. We provide a proof for completeness. 

\begin{obs}\label{O1}
Let $\Delta$ be a spherical building over the type set $I$ and let $J\subseteq I$ be self-opposite. Let $F$ be a simplex of type $J$. Then $n(J)=1$ if, and only if,  the identity in $\Pi(F)$ can be written as the product of an odd number of perspectivities. 
\end{obs}

\begin{proof}
If the identity in $\Pi(F)$ can be written as the product of an odd number of perspectivities, then, by composing this product with any even projectivity, we see that we can write any putative member of $\Pi(F)\setminus \Pi^+(F)$ as a product of an even number of perspectivities, that is, as a member of $\Pi^+(F)$, a contradiction. We conclude $\Pi^+(F)=\Pi(F)$ in this case. 

Conversely, if $\Pi^+(F)=\Pi(F)$, then consider any odd projectivity $\theta$. Our assumption implies that we can write $\theta^{-1}$ as an even projectivity. Composing those two products of perspectivities, we obtain the identity written as the product of an odd number of perspectivities.  
\end{proof}

We can now prove \cref{MR0}. 
%\begin{prop} \label{simple}Let $F$ be a simplex of a Moufang spherical building $\Delta$ of simply laced type. Let $\Aut^\dagger(\Delta)$ be the automorphism group of $\Delta$ generated by the root groups. Then the special projectivity group of $F$ coincides with the faithful permutation group induced by the stabiliser $\Aut^\dagger(\Delta)_F$ of $F$ in $\Aut^\dagger(\Delta)$ on the residue $\Res_\Delta(F)$ of $F$ in $\Delta$. \end{prop}

\emph{Proof of \cref{MR0}}.
(I) First we want to show that every even self-projectivity of $\Res(F)$ is induced by a product of elations that stabilises $F$. In fact, we are going to show that any even projectivity $$ \theta \colon \Res(F) \rightarrow \Res(T), $$ that maps $F$ to a simplex $T$, is induced by an elation. Since self-projectivities are products of projectivities, it then follows that every even self-projectivity is induced by a product of elations that stabilises $F$.

So let $\theta \colon \Res(F) \rightarrow \Res(T)$ be an even projectivity that maps $F$ to a simplex $T$. It suffices to prove the assertion for the case that $\theta$ is a product of two perspectivities. Then there exists a simplex $R$ opposite both $F$ and $T$, such that $\theta =  \proj_{T}^{R} \circ \proj_{R}^{F}$. Since $\Delta$ is Moufang, it follows with \cref{L1} that there exists an elation $g$, which maps $F$ to $T$ and fixes $R$ pointwise. For an element $f$ in $\Res(F)$, $f^{g}$ is exactly the projection of $\proj_{R} (f)$ onto $T$, since elations preserve incidence. That means $g_{|_{\Res(F)}} = \proj_{T}^{R} \circ \proj_{R}^{F}$. 

(II) Now let $g\colon \Delta \rightarrow \Delta$ be a central elation. %Since $\Aut^\dagger(\Delta)$ is generated by the central elations, we may assume that $g$ is central. 
Let $c$ be the centre of any root corresponding to $g$. Let $T$ be a simplex either containing $c$ or joinable to it. Then, by \cref{centreroot}, every vertex $u$ of $\Delta$ joinable to $T$ is contained in a root with centre $c$ (consider an apartment containing $\{c\}\cup T$ and $\{u\}\cup T$). It follows that $g$  fixes $\Res(T)$ pointwise and moves a simplex $F$ of the same type to a simplex $F^{g}$. First, we claim that the restriction $g_{|_{\Res(F)}}$ is an even projectivity from $\Res(F)$ to $\Res(F^{g})$.

Since $\Delta$ is Moufang, $\Delta$ is thick and therefore there exists a simplex $R$ in $\Delta$ opposite both $F$ and $T$. Since elations preserve incidence, the image $R^{g}$ is opposite both $F^{g}$ and $T^{g}=T$. 

Now for every $f \in \Res(F)$ we have:
\begin{equation*}
f^{g} = \proj_{F^{g}}^{R^{g}} \circ \proj_{R^{g}}^{T}  \circ \proj_{T}^{R}  \circ \proj_{R}^{F}(f), 
\end{equation*}
proving the claim. For an element $h$ of the little projective group that stabilises $F$, $h$ is a product of central elations and every such central elation gives rise to an even projectivity like above.
\qed

%%%%%%%%%%%%%%%%%%%%%%%%%%%%%%%%%%%%%%%%%%%
\section{Projective spaces}
In this section we completely settle the case of type $\mathsf{A}_r$ regarding the number $n(J)$. The proof will also contain a warm up for a general statement we will prove later on, see \cref{indlemma}. The main reason for treating this case separately, is that we can provide an elementary proof only using projective geometry independent from building-theoretic notions (we do refer to \cref{Tits}, but this can easily be verified for projective spaces). 
\begin{theorem}\label{caseAn}
For buildings of type $\mathsf{A}_r$ with type set $I$ and $J\subseteq I$, we have $n(J)=1$ if, and only if, either $J^\equiv\neq J$, or $|J|=2k$, for some $k\leq \frac{r-1}{2}$ and $J=\{1,2,\ldots,k,r-k+1,r-k+2,\ldots,r\}$, that is, $J$ is polar closed. 
\end{theorem}

\begin{proof}
First note that, for both the ``if" and the ``only if" parts, we may assume that $J$ is self-opposite. First suppose $n(J)=1$. Let $F$ and $F'$ be two opposite simplices of type $J$.  Let $j\in I\setminus J$ be minimal with respect to the Bourbaki labelling of the diagram and let $v$ be a vertex of type $j$ incident to $F$. Since $J$ is self-opposite,  $j\leq\frac{r}{2}$. Then according to \cref{Tits}, the type $j'$ of $\proj^{F}_{F'}(v)$ is the opposite type in $\Res_{\Delta}(F')$ of type $r+1-j$ (which belongs to $I\setminus J$ since $I\setminus J$ is self-opposite). If $n(J)=1$, we should have $j=j'$. This is only possible, if the integer interval $[j,r+1-j]$ belongs to $I\setminus J$. Putting $k=j-1$, we obtain the ``only if'' part of the statement.

Now we show the ``if'' part. We establish the identity projectivity as a product of three perspectivities. Let $F$ be any simplex of type $J$. Suppose $F=\{U_i\mid i\in J, \dim U_i=i-1\}$. Note that, since $F$ is a simplex, $U_i\leq U_j$ for $i\leq j$, with $i,j\in J$. Select a simplex $F'$ opposite $F$ and set $F'=\{U'_i\mid i\in J, \dim U'_i=i-1\}$. Choosing a basis $\{p_0,p_1,\ldots,p_r\}$ well, we may assume $U_i=\<p_0,\ldots,p_{i-1}\>$ and $U_i'=\<p_r,p_{r-1},\ldots,p_{r-i+1}\>$. Let, for $0\leq i\leq k-2$, the point $q_i$ be an arbitrary point on the line $\<p_i,p_{r-i}\>$ distinct from both $p_i$ and $p_{r-i}$. Define $U''_i=\<q_0,\ldots,q_{i-1}\>$, for $1\leq i\leq k-1$, and $U''_i=\<U_{r-i+1},p_{r-i+1},\ldots,p_{i-1}\>$. Then the simplex $F''=\{U''_i\mid i\in J\}$ is easily checked to be opposite both $F$ and $F'$. Let $W$ be an arbitrary subspace of dimension $k$ containing $U_k$ and contained in $U_{r-k+1}$.  Then $W$ is generated by $U_k$ and a point $p\in\<p_k,\ldots,p_{r-k}\>$. The point $p$ belongs to $U'_{r-k+1}\cap U''_{r-k+1}$. Consequently $\proj^{F}_{F'}(W)=\<U'_k,p\>=:W'$, $\proj^{F'}_{F''}(W')=\<U''_k,p\>=:W''$ and $\proj^{F''}_{F}(W'')=W$. This implies that $\proj^{F''}_{F}\circ\proj^{F'}_{F''}\circ\proj^{F}_{F'}$ is the identity and, by \cref{O1},  the assertion is proved.
\end{proof}

\section{Proof of \cref{MR1}}\label{3cycle}

The following lemma is basically the gate property of buildings.

\begin{lem}\label{gate}
Let $\Delta$ be a spherical building over the type set $I$ and let $F_J$ be  a simplex of type $J\subseteq I$.  Let $K\subseteq J$ and let $F_K$ be the face of $F_J$ of type $K$. Let $F_J^*$ be opposite $F_J$ and let $F_K^*\subseteq F_J^*$ be opposite $F_K$. Set $F'_J:=\proj_{F_K}(F_J^*)$. Let $C\supseteq F_J$ be a chamber. Then \[\proj_{F_J^*}(C)=\proj_{F_J^*}(\proj_{F'_J}(C)).\] 
\end{lem}

\begin{proof}
This follows from the gate property of residues. Since $F'_J=\proj_{F_K}(F_J^*)$, \[F'_J\subseteq \proj_{F_K}(\proj_{F_J^*}(C)).\] The latter is on every minimal gallery joining $\proj_{F_J^*}(C)$ with $C$ and hence equals $\proj_{F_J'}(C)$. The assertion follows. 
\end{proof}

In the next lemma we use the following terminology. A triple of pairwise opposite simplices  $S_1,S_2,S_3$ is called a \emph{projective $3$-cycle}, if  $\proj^{S_3}_{S_1}\circ\proj^{S_2}_{S_3}\circ \proj^{S_1}_{S_2}=\id$. Note that, if the triple $S_1,S_2,S_3$ is a projective $3$-cycle, then so is the triple $S_i,S_j,S_k$, with $(i,j,k)$ any permutation of $(1,2,3)$. Also,  if $S_1, S_2, S_3$ form a projective $3$-cycle, then they all have the same self-opposite type, say $J$, and projections between two opposite simplices of type $J$ are type-preserving. 

\begin{lem}\label{indlemma}
Let $\Delta$ be a spherical building over the type set $I$ and let $S_1,S_2,S_3$ be  a projective $3$-cycle of type $J\subseteq I$.  Let $K\subseteq I\setminus J$ be such that, for each pair of $S_3$-opposite simplices $T_3, T_3'\in\Res(S_3)$, there exists a simplex $T_3''$ such that $T_3,T_3',T_3''$ is a projective $3$-cycle in $\Res_\Delta(S_3)$. Then $n(J\cup K)=1$. More exactly, if $T_1$ is a simplex of type $K$ adjacent to $S_1$, then there exist simplices $T_2'\sim S_2$ and $T_3''\sim S_3$ of type $K$ such that the triple $S_1\cup T_1$, $S_2\cup T_2'$, $S_3\cup T_3''$ is a projective $3$-cycle.  
\end{lem}

\begin{proof}
Let $T_1$ be a simplex of type $K$ adjacent to $S_1$. We want to write the identity in $\Res_\Delta(S_1\cup T_1)$ as the product of three projections.

Since $S_1,S_2,S_3$ is a projective $3$-cycle, $\proj^{S_1}_{S_2}T_1=\proj^{S_3}_{S_2}T_3$, where $T_3=\proj^{S_1}_{S_3}T_1$. Hence we have

\begin{align*}
T_{2} &= \proj_{S_{2}} (T_{1}) = \proj_{S_2}(T_3), \\
T_{3} &= \proj_{S_{3}} (T_{2}) = \proj_{S_{3}} (T_{1}),\\
T_{1} &= \proj_{S_{1}} (T_{3}) =\proj_{S_1}(T_2).
\end{align*}

Let $T_{2}'$ be a simplex locally opposite $T_{2}$ at $S_{2}$. Then, by \cref{Tits}, the simplices $T_{1}$ and $T_{2}'$ are opposite in $\Delta$. Set $T_{3}'=\proj_{S_3}(T_2')$. Then $T_{3}'$ is opposite $T_{2}$ in $\Delta$ (again by \cref{Tits}). Since $T_3=\proj_{S_3}(T_2)$, this implies, again using \cref{Tits}, that $T_3'$ is locally opposite $T_{3}$ at $S_3$. Our assumption permits to choose a simplex $T_{3}''\sim S_3$ of type $K$  such that $T_{3}, T_3',T_3''$ is a projective $3$-cycle in $\Res_\Delta(S_3)$. Since, in particular, $T_3''$ is locally opposite both $T_3$ and $T_3'$ at $S_3$, we have similarly as before (using \cref{Tits}) the following opposite relations:

\begin{align*}
T_{1} &\equiv T_{2}'\equiv T_3''\equiv T_1, \\
T_{3} &\equiv_{S_{3}} T_{3}' \equiv_{S_{3}} T_{3}''\equiv_{S_3} T_{3}. \\
\end{align*}

Let $v_{1}$ be an arbitrary vertex adjacent to $S_{1}\cup T_1$. We want to see that if we project $v_{1}$ first onto  $S_{2}\cup T_2'$, then onto $S_{3}\cup T_3''$ and back to $S_{1}\cup T_1$, then we get $v_{1}$ again. Define:

\begin{align*}
v_{2} & := \proj_{S_{2} \cup T_{2}} (v_{1}), \\
v_{3} & := \proj_{S_{3} \cup T_{3}} (v_{1}) = \proj_{S_{3} \cup T_{3}} (v_{2}), \\
v_{2}' & := \proj_{S_2\cup T_{2}'} (w_{2}); \text{ then }v_2' = \proj_{S_{2} \cup T_{2}'} (v_{3}), \\
v_{3}' &:= \proj_{S_{3}} (v_{2}'); \text{ then }v_3'\sim T_3', \\
v_{3}'' & := \proj_{S_{3} \cup T_{3}''} (v'_{2}) \\
\end{align*}

We have that $v_{i}$ is adjacent to $T_{i}$ for $i \in \{ 1, 2, 3\}$, that $v_{j}'$ is adjacent to  $T_{j}'$ for $j \in \{2,3\}$ and that $v_{3}''$ is adjacent to $T_{3}''$, since incidences are preserved under projection.

By \cref{gate}, the projection of $v_{2}'$ from $S_{2} \cup  T_{2}' $ onto $S_{3} \cup T_{3}'' $ is the same as the projection onto $T_{3}''$ of the projection of $v_{2}'$ from $S_{2}$ onto $S_{3}$ and this is the same as $\proj_{T_{3}''} (v_{3}')$ (namely $v_{3}''$).

If we project $v_{3}'$ onto $T_{2}$, we get the vertex $v_{2}'$. If we project further onto $T_{2}$, we get the vertex $v_{2}$. The converse shows that $v_{3}$ maps to $v_{3}'$ under the projection locally at $S_3$ from $T_3$ to $T_3'$. 

Now the projection of $v_1$ onto $S_3\cup T_3''$ is obtained by first projecting onto $S_3$ (and this is $v_3$), and then projecting $v_3$ locally at $S_3$ onto $T_3''$. But since the triple $T_3,T_3',T_3''$ is a projective $3$-cycle, we have locally at $S_3$:
\[
 \proj_{T_{3}''} (v_{3}) = \proj_{T_{3}''}(\proj_{T_{3}'} (v_{3}))  = \proj_{T_3''}v_3'=v_{3}'',
\]
which shows that the triple $S_1\cup T_1$, $S_2\cup T_2'$, $S_3\cup T_3''$ is a projective $3$-cycle. This concludes the proof of the lemma. 
\end{proof}

In view of \cref{indlemma}, and in order to prove \cref{MR1}, it suffices to show that, for any irreducible building $\Delta$, there exists a triple of simplices of polar type which is a projective $3$-cycle. 
\begin{prop}\label{polartypen=1}
Let $\Delta$ be a spherical building. Let $F$ and $F'$ be two opposite simplices of polar type. Then $F$ and $F'$ are contained in a projective $3$-cycle. 
\end{prop}
\begin{proof}
Let $C$ be a chamber containing $F$, let $\Sigma$ be an apartment containing $C$ and $F'$, let $\alpha$ be the root in $\Sigma$ with centre $F$ (and so containing $C$) and let $C'=\proj_{F'}(C)$. Then $F'$ is the centre of the opposite root $-\alpha$ of $\alpha$ in $\Sigma$. Let $\theta\in U_\alpha$ be a non-trivial root elation and set $F''={F'}^\theta$. Let $(C_0,C_1,\ldots,C_\ell)$ be a minimal path in the chamber graph of $\Delta$ connecting $C=C_0$ with $C'=C_\ell$. By symmetry, $\ell=2k$ is even and $C_0,\ldots,C_k$ all belong to $\alpha$, whereas $C_{k+1},\ldots,C_{2k}$ belong to $-\alpha$. The root $(-\alpha)^\theta$ has centre $F''$ and contains $C_{k+1}^\theta, \ldots, C_{2k}^\theta=:C''$. Moreover, since $\theta$ fixes $\partial\alpha=\partial(-\alpha)$ pointwise, the union $(-\alpha)\cup(-\alpha)^\theta$ is an apartment and the chambers $C_{k+1}$ and $C_{k+1}^\theta$ are adjacent. Hence $F''$ is opposite $F'$ and $\delta(C,C')=\delta(C,C'')=\delta(C',C'')$. All this yields
\[\proj_{F'}^{F''}(C'')=C'.\] This shows that $\{F,F',F''\}$ is a projective $3$-cycle. 
\end{proof}

\emph{Proof of \cref{MR1}.} If $J\neq J^\equiv$, then there are no odd self-projectivities and $n(J)=1$. If $J$ is the polar type, then $n(J)=1$ by \cref{polartypen=1}, and if $J$ is polar closed, then $n(J)=1$ by \cref{indlemma}. \qed 

\section{Proof of \cref{MR2}}
The following is a direct consequence of  \cref{O1}.
\begin{lem}\label{newlemma}
Let $J \subseteq I$ be such that $n(J) = 1$ and $J = J^\equiv$. Then the opposition relation in $I\setminus J$ coincides with the restriction to $I\setminus J$ of the opposition relation in $I$.
\end{lem}

We then prove \cref{MR2} by verifying that, as soon as $J$ is not polar closed and $I\setminus J$ contains a connected component of rank at least 2, then for some connected component of $I\setminus J$, the opposition relation of that component does not coincide with the global opposition relation. %\cref{Tits} implies that a single perspectivity $\proj^{F}_{F'}$, with $F,F'$ simplices of type $J$, does not preserve types. This implies that every member of $\Pi(F)$ which is the product of an odd number of perspectivities is a duality in $\Res_\Delta(F)$, and hence cannot be the identity. \cref{O1} then yields $n(J)=2$.  
We first treat the exceptional cases and then the infinite class of type $\mathsf{D}_n$. The case $\mathsf{A}_n$ follows from \cref{caseAn}. 
\subsection{Type $\mathsf{E_6}$}

Out of the $2^6-2= 62$ possible types of a non-empty non-maximal simplex, there are exactly $2^4-2=14$ self-opposite ones. Out of these 14, there are precisely seven for which $I\setminus J$ has a connected component of rank at least 2. We present the possibilities pictorially, colouring the vertices of types in $J$ black. For the other seven $I\setminus J$ is the union of isolated vertices.

$\dynkin E{o*oooo}$ and $\dynkin E{**ooo*}$ are polar closed.

$\dynkin E{*oooo*}$: According to \cref{Tits}, types $3$ and $5$ are interchanged by a perspectivity. 

$\dynkin E{*o*o**}$ and $\dynkin E{oo*o*o}$: Opposition in $\mathsf{D_4}$ is trivial, whereas opposition in $\mathsf{E_6}$  interchanges types $3$ and $5$.

$\dynkin E{ooo*oo}$ and $\dynkin E{o*o*oo}$: Opposition in $\mathsf{E_6}$ interchanges the two rank 2 residues. 

\subsection{Type $\mathsf{E_7}$} All of the $2^7-2=126$ possible types of non-empty non-maximal simplices are self-opposite, as opposition is trivial here. There are 18 polar closed types of which only three with a residue containing a connected component of rank at least 2. These components are of types $\mathsf{D_4}$ and $\mathsf{D_6}$; the three cases are 
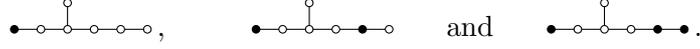
\[\dynkin E{*oooooo}, {}\hspace{1cm}{}\dynkin E{*oooo*o} \hspace{.5cm}\mbox{ and }\hspace{.5cm}\dynkin E{*oooo**}.\]
Now, the only connected subdiagrams of size at least 2 admitting trivial opposition are precisely the ones of types $\mathsf{D_4}$ and $\mathsf{D_6}$. The above choices for $J$ are the only ones for which $I\setminus J$ has a connected component of  size at least 2 admitting trivial opposition. In all other cases it follows from \cref{newlemma} that $n(J)=2$. 

\subsection{Type $\mathsf{E_8}$} Here opposition is also trivial. There are 19 polar closed types of which only four with a residue containing a connected component of rank at least 2. These components are of types $\mathsf{D_4}$, $\mathsf{D_6}$ and $\mathsf{E_7}$; the four cases are 
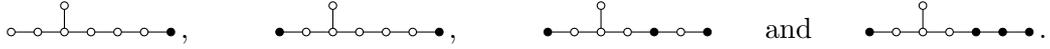
\[\dynkin E{ooooooo*},\hspace{1cm} \dynkin E{*oooooo*}, \hspace{1cm}\dynkin E{*oooo*o*} \hspace{.5cm}\mbox{ and }\hspace{.5cm}\dynkin E{*oooo***}.\]
There is actually exactly one more type with a residue of rank $4$ admitting trivial opposition:

$\dynkin E{*oooo**o}$: Here the unique connected component $K=\{2,3,4,5\}$ of $I\setminus J=\{2,3,4,5,8\}$ has the property that $I\setminus K=\{1,6,7,8\}=\dynkin E{*oooo***}$ is polar closed. 

Since all other connected subdiagrams of  size at least $2$ are either of type $\mathsf{A_2},\ldots,\mathsf{A_7}$, $\mathsf{D_5}$, $\mathsf{D_7}$ or $\mathsf{E_6}$, we see that for all other types $J$ such that $I\setminus J$ has a connected component of size at least $2$, we have $n(J)=2$.
\subsection{Type $\mathsf{D}_n$, $n\geq 4$}
Obviously, the only connected subdiagrams of size at least 2 of a diagram of type $\mathsf{D}_n$, $n\geq 4$, where opposition agrees with the opposition in $\mathsf{D}_n$ are of type $\mathsf{D}_{n-2k}$, for $k\in\mathbb{N}$ such that $n-2k\geq 3$. So a counterexample $J$ to the assertion has $\max J=n-(n-2k)=2k$ and the connected component $K$ of size at least 2 of $I\setminus J$ is unique. Clearly, $I\setminus K$, which consists of the vertices of types $1,2,\ldots,2k$, is polar closed (indeed, consider the ordering $2,1;4,3;\ldots;2k,2k-1$).

\section{Projectivity Groups of panels---Proof of \cref{MR3}}

\subsection{A basic lemma} The next lemma will enable us to pin down the special and general projectivity groups for residues which have the full linear group as respective projectivity group in a residue.  
\begin{lem}\label{gate2}
Let $\Delta$ be a spherical building over the type set $I$ and let $F_K$ be  a simplex of type $K\subseteq I$.  Let $K\subseteq J\subset I$ and let $F_J$ be a simplex of type $J$ containing $F_K$. Let $\Pi^+_K(F_J)$ be the special projectivity group of $F_J\setminus F_K$ in $\Res_\Delta(F_K)$. Then $\Pi^+_K(F_J)\leq \Pi^+(F_J)$.
\end{lem}

\begin{proof}
Let $F_J'$ and $F_J''$ be two simplices containing $F_K$ such that $F_J'\setminus F_K$ is opposite both $F_J\setminus F_K$ and $F_J''\setminus F_K$ in $\Res_\Delta(F_K)$. We have to show that the product of the two perspectivities in $\Res_\Delta(F_K)$ from $\Res_\Delta(F_J)$ to $\Res_\Delta(F_J')$, subsequently to $\Res_\Delta(F_J'')$ coincides with the product of two perspectivities in $\Delta$. To that aim, let $F_K^*$ be a simplex in $\Delta$ opposite $F_K$, and let $F_J^*$ be the projection of $F_J'$ onto $F_K^*$ (hence $F_J^*=\proj_{F_K^*}^{F_K'}(F_J')$). 
 
Let $C$ be any chamber containing $F_J$. Set   
\begin{align*}
C' &= \proj_{F_J'} (C), \\
C'' &= \proj_{F_J''} (C') = \proj_{F_J''}(\proj_{F_J'} (C)),\\
C^* &= \proj_{F_J^*}(C').
\end{align*}
Then, according to \cref{gate}, we have 
\[C^*=\proj_{F_J^*}(C) \mbox{ and } C''=\proj_{F_J''}(C^*),\]
which implies that $C''$ is indeed equal to the image of $C$ under the product of two perspectivities in $\Delta$. 
\end{proof}
Recall that an automorphism of a spherical building $\Delta$ of simply laced type is called \emph{linear}, if it belongs to $\PGL_{r+1}(\L)$ in case $\Delta$ corresponds to $\PG_r(\L)$, or if it belongs to the linear algebraic group corresponding to the building if $\Delta$ has type $\mathsf{D}_r$, $r\geq 4$, or $\mathsf{E_6,E_7,E_8}$. (For a more precise definition using the corresponding Chevalley group, see \cref{algapp}.) The next result is an immediate consequence of \cref{gate2}. 
\begin{coro}\label{corgate}
Let $\Delta$ be a spherical building over the type set $I$ and let $F_K$ be  a simplex of type $K\subseteq I$.  Let $K\subseteq J\subset I$ and let $F_J$ be a simplex of type $J$ containing $F_K$.  Let $\Pi^+_K(F_J)$ be the special projectivity group of $F_J\setminus F_K$ in $\Res_\Delta(F_K)$. Suppose that $\Pi^+_K(F_J)$ is the full linear type preserving automorphism group of $\Res_\Delta(F_J)$. Then $\Pi^+(F_J)$ also coincides with the full linear type preserving automorphism group of $\Res_\Delta(F_J)$. 
\end{coro}

\subsection{End of the proof} Now \cref{MR3} follows from \cref{corgate} because every vertex of the Coxeter diagram of a simply laced irreducible spherical building of rank at least 3 is contained in a residue isomorphic to the building of a projective plane over some skew field $\L$, and in a projective plane the special projectivity group of a line is $\PGL_2(\L)$ acting naturally on $\PG(1,\L)$. 
%%%%%%%%%%%%%%%
%%%%%%%%%%%%%%%
%%%%%%%%%%%%%%%
%%%%%%%%%%%%%%%

\section{General and special projectivity groups of irreducible residues of rank at least $2$} \label{allgroups}

In this section we determine the exact projectivity groups for irreducible residues. 
We start with the algebraic approach to the special projectivity groups. 

\subsection{The special projectivity groups}\label{algapp}

The arguments in this section were generously suggested to us by the referee.

%\subsubsection{Root groups of Moufang buildings}

Let $\Delta$ be an irreducible spherical Moufang building with associated Coxeter system $(W, S)$.  Let $C$ be a
chamber and let $\Sigma$ be an apartment containing $C$. For each $J\subset S$ let $\bar{J} := S\setminus J$ and let 
$R_J$ be the $J$-residue containing $C$. Furthermore, let $R_J'$ denote the unique residue of $\Sigma$ that is opposite $R_J$.

For each $s$ let $T_s$ denote the group of all automorphisms of $\Delta$ stabilising $\Sigma \cup R_{\bar{s}}$ pointwise. The groups $T_s$, $s\in S$, obviously normalise each other and we put $T:=\langle T_s\mid s\in S \rangle \leq \mathrm{Aut}^\circ(\Delta)$. For a root $\alpha$ in $\Sigma$, let $U_\alpha$ be the  root group associated with $\alpha$. As usual, we denote by $-\alpha$ the opposite root. 

Furthermore, let $\alpha_s$ be the root of $\Sigma$ containing $C$ such that $R_s$ is on the boundary of $\alpha_s$, %let $U_s$ and $U_{-s}$ be the root groups associated with $\alpha_s$ and $\alpha_{-s}$, respectively, 
and let $L_s:=\langle U_{\alpha_s},U_{-\alpha_s} \rangle$.

We abbreviate $G^{\dagger}:=\Aut^\dagger(\Delta)= \langle U_\alpha\mid \alpha \mbox{ is a root of } \Sigma \rangle \leq \mathsf{Aut}^\circ(\Delta)$. Let $H\leq G^\dagger$ 
be the pointwise stabilizer of $\Sigma$ in $G^\dagger$ and let $H_s:=H\cap L_s$ for $s\in S$. Again, the $H_s$, $s\in S$, normalise each other 
and $H=\langle H_s\mid s\in S\rangle$.

%\subsubsection{Root groups of Chevalley groups}

We now restrict to the case of Chevalley groups over fields, that is, $G^\dagger$ is a Chevalley group over a field $\K$, and we assume that the corresponding spherical building $\Delta$ is irreducible.  

For each $s\in S$ the group $T_s$ acts regularly on $R_s\setminus \Sigma$, which yields an isomorphism 
$\theta_s:\mathbb{K}^*\to T_s$ given by $\lambda\mapsto \theta_s(\lambda)$. These isomorphisms provide a canonical isomorphism
$T\cong (\K^*)^{|S|}$.

Also, for each $s\in S$ there is a canonical homomorphism $h_s:\K^*\to T$ given by $\lambda \mapsto h_s(\lambda)$
whose image is $H_s$. 

Note that $H\leq T$. By \cite[Lemma~27]{Ste:68}, the group $H$ corresponds to the root lattice of $G^\dagger$ and, if the root lattice coincides with the weight lattice of $G^\dagger$, then the groups $T$ and $H$ coincide (see for instance \cite[\S1.11]{Car:93}). In general, 
the group $\widehat{G}=G^\dagger T$ is the group belonging to the type-preserving linear algebraic group corresponding to the building $\Delta$ (see previous section). 

Let $J\subseteq S$ (and we may think of $J$ as corresponding to a connected subdiagram of the Coxeter diagram of $\Delta$). Let $L_J$ be the corresponding Levi subgroup, that is, $L_J$ is the stabiliser in $G^\dagger$ of $R_J\cup R_J'$. By Theorem \ref{MR0}, $\Pi^+(R_J)$ is the image of $L_J$ in $\Aut^\circ(R_J)$. Setting $T_J:=L_J\cap T$, we find $L_J=\Aut^\dagger(R_J)T_J$. So, to obtain $\Pi^+(R_J)$ we have to determine $T_J$. 
This boils down to understanding the action of $h_s(\K^*)$ on $R_J$, for all $s\in S\setminus J$ (for $s\in J$, this action is already in $\Aut^\dagger(R_J)$). If $s\in S\setminus J$ is not connected to $J$ in the Coxeter diagram, then
the action of $H_s$ on $R_J$ is trivial, as both $U_{\alpha_s}$ and $U_{-\alpha_s}$ pointwise fix $R_J$, and $H_s$ is inside $L_s$. If $s\in S\setminus J$ is connected to a vertex $j\in J$ (and we write $s\sim j$), then $j$ is uniquely determined. We claim that the action induced by $H_s$ on $R_J$ is the action of $T_j$. Indeed, since by uniqueness of $j\sim s$, both $H_s$ and $T_j$ pointwise fix $R_{J\setminus\{j\}}$, since $T_j$ acts faithfully on $R_j$, and $H_s\leq L_s$, it suffices to look in $R_{\{s,j\}}$, which is a projective plane. We may choose coordinates such that $R_j$ corresponds to line $(*,*,0)$ and $R_s$ to the line pencil with vertex the point $(1,0,0)$ (thus, $C\cap R_{\{s,j\}}=\{(1,0,0),(*,*,0)\}$). Then \[H_s/R_{\{s,j\}}=\left\{\begin{pmatrix} 1 & & \\  & k &  \\  &  & k^{-1}\end{pmatrix}\left| \right.k\in\K^\times\right\}\mbox{ and }T_j/R_{\{s,j\}}=\left\{ \begin{pmatrix} \ell & & \\  & 1 &  \\  &  & 1\end{pmatrix}\mid \ell\in\K^\times\right\}.\] Restricting to the first two coordinates and putting $\ell=k^{-1}$, the claim follows. Hence $\Pi^+(R_J)$ is generated by $\Aut^\dagger(R_J)$ and all $T_j$, $j\in J$, such that there exists $s\in S\setminus J$ with $s\sim j$. It remains to determine $\Aut^\dagger(R_J)T_j$ for these values of $j\in J$ (and the groups they generate for suitable different $j$). These follow from straight forward computations in the corresponding Chevalley groups. Let us explain the main examples such that the reader can verify the tables.

We first introduce some notation.

\begin{nota}
  For each $a\in\mathbb{N}$, let \[\PSL_n(\K,a):=\{M\in\GL_n(\K)\mid \det M=k^a, k\in\K\}.\Sc_n(\K)/\Sc_n(\K),\] where $\Sc_n(\K)$ is the group of all scalar matrices over $\K$. We get $\PGL_n(\K)$ by putting $a=1$ and $\PSL_n(\K)$ by putting $a=n$. Note that we can always choose $a$ as a divisor of $n$ since $\PSL_n(\K,a)=\PSL_n(\K,g)$, with $g=\gcd(a,n)$. 
 \end{nota} 

\begin{compactenum}[$(1)$]
\item Let $R_J$ be the building of a projective space $\PG(d,\K)$.  We relabel $J=\{1,2,\ldots,d+1\}$ according to Bourbaki \cite{Bou:68}. Then, with respect to the standard chamber in $\PG(d,\K)$, we have \[T_j=\left\{\begin{pmatrix} kI_j & \\ & I_{d-j+1}\end{pmatrix}\mid k\in\K^\times \right\},\] where $I_\ell$ is the $\ell\times\ell$ identity matrix. It follows that $\PSL_{d+1}(\K)T_j=\PSL_{d+1}(\K,j)=\PSL_{d+1}(\K,g)$, with $g=\gcd(d+1,j)$. This has the following consequences:
\begin{compactenum}[$(i)$]
\item If some endpoint of the subdiagram of the diagram of $\Delta$ corresponding to $R_J$ has a neighbour outside $J$, then $\Pi^+(R_J)$ is $\PGL_{d+1}(\K)=\widehat{G}$. 
\item If $j\in J$ corresponds to a Bourbaki label that is relatively prime to $d+1$, and it is connected to a vertex outside $J$, then, again, we have $\Pi^+(R_J)=\PGL_{d+1}(\K)=\widehat{G}$. (Note $(i)$ is a special case of this.)
\item There are precisely three cases which are not applicable to either $(i)$ or $(ii)$. The first one is $J=\{1,2,.\ldots,n-1\}$ in a building of type $\mathsf{D}_n$, with $n$ even. Here, $j=n-2$ and we get $\PSL_n(\K,2)$. The second one is $J=\{1,3,4,5,6\}$ in $\mathsf{E_6}(\K)$. Here, the node $4$ is joint to a vertex outside $J$, namely $2$. The former node has in Bourbaki labelling for $\mathsf{A_5}$  label $3$ and hence $\Pi^+(R_J)=\PSL_6(\K,3)$. The third situation is when $J=\{2,4,5,6,7\}$ in $\mathsf{E_7}(\K)$.  Here, the second node is joint to an outside node, leading to $\Pi^+(R_J)=\PSL_6(\K,2)$. 
\end{compactenum}
\item Let $R_J$ be the building of type $\mathsf{D}_n$ over $\K$.  Due to the existence of a branching vertex of valency $3$ in the Dynkin diagram, only the end vertices can be joint to  vertices outside this diagram when viewed as a subdiagram of the Dynkin diagram of $\Delta$. So, we are only interested in $G^\dagger T_1$, $G^\dagger T_{n-1}$, $G^\dagger T_n$ and the groups generated by any combination of those.    Referring forward to \cref{Q0} and \cref{Q}, it follows from   \cref{PGO} that $G^\dagger T_1=\PGO_{2n}^\circ(\K)$, $G^\dagger T_nT_{n-1}=\overline{\PGO}_{2n}^\circ(\K)=\widehat{G}$ (but $T_n=T_{n-1}$ if $n$ is odd), and  $G^\dagger T_n=G^\dagger T_{n-1}=\overline{\mathsf{P\Omega}}_{2r}(\K)$ if $n$ is even.  \cref{PGO}$(iii)$ now implies $G^\dagger T_1T_n=\overline{\PGO}_{2n}^\circ(\K)=\widehat{G}$. All this implies that there are exactly two cases where $\Pi^+(R_j)$ is not the full linear group $\widehat{G}$. 
\begin{compactenum}[$(i)$]
\item $\Delta$ is the building of type $\mathsf{D}_{n+1}$ over $\K$ and $S\setminus\{J\}=\{1\}$. Then the above implies $\Pi^+(R_J)=\PGO_{2n}^\circ(\K)$.
\item $\Delta$ is the building of type $\mathsf{E_7}$ over $\K$ and $J=\{2,3,4,5,6,7\}$. It follows from the above immediately that $\Pi^+(R_J)=\overline{\mathsf{P\Omega}}_{12}(\K)$. 
\end{compactenum}
\item  Let $R_J$ be the building of type $\mathsf{E_6}$ over $\K$. We are only interested in $G^\dagger T_1$, which equals $\widehat{G}$ by \cref{E6T1}. Note that, in this case, the geometric approach obtains $\Pi^+(R_J)=\widehat{G}$ in a quite different way, namely using \cref{polE6} and symplectic polarities.  
\item Let $\Delta$ be the building of type $\mathsf{E_8}$ over $\K$. Then $\widehat{G}=G^\dagger$, as the weight lattice and root lattice coincide. Hence for every $J\subseteq S$ we have $\Pi^+(R_J)$ is the full linear group. 
\end{compactenum}

\bigskip

We now proceed with the geometric proof (including the determination of the general projectivity groups) and start with some general results.
\subsection{General considerations} 
The \emph{fix set} of an automorphism $\rho$ of a building $\Delta$ is the set of simplices fixed under $\rho$. Two automorphisms $\varphi_1$ and $\varphi_2$ are called \emph{congruent} if their fix sets are isomorphic (using a type preserving automorphism of the building). Conjugate (with respect to a type-preserving automorphism) automorphisms are examples of congruent automorphisms. Clearly, congruence is an equivalence relation and an equivalence class $\Pi$ is called a  \emph{geometric} set of automorphisms.  We are going to use this notion only for rather large fix sets. We will mention the examples that we will need in our proofs at appropriate places (see \cref{exDnBn},~\ref{exDn} and~\ref{exE6}). Here, we provide two examples.

\begin{exm}\label{exAn}
Let $U$ and $U'$ be two complementary subspaces of the projective space $\PG(r,\K)$ (that is, $U$ and $U'$ generate the whole space and are disjoint). Then the set of nontrivial collineations pointwise fixing $U\cup U'$, together with all their conjugates, is a geometric set of automorphisms of $\PG(r,\K)$. 
\end{exm}

\begin{exm}\label{exAnCn}
The set of all symplectic polarities of a given projective space is a geometric set of dualities. (This follows from the fact that such polarities are characterised by the property that each point is mapped onto a hyperplane it is contained in.)
\end{exm}

 \begin{lem}\label{triangles}
 Let $\Delta$ be a spherical building over the type set $I$ and let $J\subseteq I$ be a self-opposite type. Suppose that for each quadruplet of simplices of type $J$, there exists a simplex of type $J$ opposite all  the given simplices. Let $F, F', F''$ be three pairwise opposite simplices of type $J$ and denote by $\theta_0$ the projectivity $F\per F'\per F''\per F$. Denote with $\Pi_3(F)$ the set of all self-projectivities of $F$ of length $3$ and suppose that $\Pi_3(F)$ is geometric. %Suppose also that the set $G_0:=\{\theta_0^{-1}\theta\mid \theta \in\Pi_3(F)\}$ is a group. 
 Then $\Pi(F)=\<\Pi_3(F)\>$ and $\Pi^+(F)=\<\theta_0^{-1}\theta\mid \theta \in\Pi_3(F)\>$. 
 \end{lem}
 \begin{proof}
It is clear that the said groups are subgroups of the respective projectivity groups. Now we claim that every self-projectivity of $F$ of length $\ell$ is the product of $\ell\mod 2\Z$ members of $\Pi_3(F)$. First note that, if $F^*$ is a simplex of type $J$ and $\theta:\Res(F)\to\Res(F^*)$ is an isomorphism, then $\theta\Pi_3(F^*)\theta^{-1}=\Pi_3(F)$, by the fact that $\Pi_3(F)$ is geometric. 

Now let $F=F_0\per F_1\per \cdots\per F_\ell=F$ be a self-projectivity of length $\ell$. Suppose $\ell\geq 5$. Let $G$ be opposite all of $F_0,F_1,F_2$ and $F_3$.  Denote $\theta_i:F_0\per G\per F_i$ and $\rho_i=F_i\per G\per F_{i-1}\per F_i$, $i=1,2,3$. Then we have \[F_0\per F_1\per F_2\per F_3=\theta_1\rho_1\theta_1^{-1}\cdot\theta_2\rho_2\theta_2^{-1}\cdot\theta_3\rho_3\theta_3^{-1}\cdot\theta_3.\]
Hence we can replace $F_0\per F_1\per F_2\per F_3$ by the product of three members of $\Pi_3(F)$ and the projectivity $\theta_3=F\per G\per F_3$ of length 2. So, the claim will follow inductively, if we show it for $\ell=4$, that is, in the above we have the additional perspectivity $F_3\per F_0$. Hence we have, with the same notation, and denoting additionally $\rho_4=F_0\per G\per F_3\per F_0$, which belongs to $\Pi_3(F)$,
 \[F_0\per F_1\per F_2\per F_3\per F_0=\theta_1\rho_1\theta_1^{-1}\cdot\theta_2\rho_2\theta_2^{-1}\cdot\theta_3\rho_3\theta_3^{-1}\cdot\rho_4,\] which is a product of four, hence an even number of, elements of $\Pi_3(F)$. Now the assertions are clear, noting that every product $\theta_1\theta_2$ of members of $\Pi_3(F)$ can be written as the product $(\theta_0\theta_1^{-1})^{-1}\cdot(\theta_0\theta_2)$ of two automorphisms of the form $\theta_0\theta$, where $\theta\in\Pi_3(F)$. 
 \end{proof}
 
 We will usually apply this lemma to the case where all members of $\Pi_3(F)$ are type-interchanging involutions, and so $\Pi^+(F)$ will also be the intersection of $\Pi(F)$ with the group of type preserving collineations. It is precisely this method that provides an alternative ``reason'' for the special projectivity groups to be what they are versus the algebraic approach explained in \cref{algapp}. 
 
 In \cref{triangles}, there is the condition that we find a simplex opposite four given simplices. It is well-known that one can find a chamber opposite two given chambers, see Proposition~3.30 in \cite{Tits:74}. We can generalise this so that the condition in \cref{triangles} becomes automatic for buildings with thickness at least $5$; for the simply laced case this just means that the building is not defined over the fields $\F_2$ or $\F_3$.  
 
 We say that a building \emph{has thickness at least $t$} if every panel is contained in at least $t$ chambers. The following generalises Proposition~3.30 of \cite{Tits:74}. The proof is also a rather obvious generalisation. 
 
 \begin{prop}\label{3.30}
 If a spherical building has thickness at least $t+1$, then there exists a chamber opposite $t$ arbitrarily given chambers.   In particular, there exists a vertex opposite $t$ arbitrarily given vertices of the same self-opposite type. 
 \end{prop}
 
 \begin{proof}
We will prove the claim by induction. First consider the case that $t=2$. Then the condition that every panel is contained in at least $t+1 = 3$ chambers is equivalent to $\Delta$ being a thick building and the assertion follows with Proposition~3.30 of \cite{Tits:74}. Now suppose $t > 2$. Suppose we know we can find a chamber opposite $t-1$ given chambers. Let $C_{1}, \dots, C_{t-1}$ be $t-1$ different chambers in $\Delta$ and let $C_{t}$ be another chamber in $\Delta$. Among all the chambers in $\Delta$ opposite to each $C_{i}$, $i \in \{1, \dots, t-1\}$, let $E$ be a chamber with maximal distance to $C_{t}$. Assume that $C_{t}$ and $E$ are not opposite. Then $\dist(C_{t},E) \neq \diam\Delta$. Let $\Sigma$ be an apartment containing both $C_{t}$ and $E$. With Proposition 2.41 of \cite{Tits:74}, it follows that there exists a face $A$ of codimension $1$ of $E$, such that $E = \proj_{A} (C_{t})$. 

Since every panel is contained in $t+1$ chambers, we can find a chamber $E'$ having $A$ as a face that is not equal to $E$ and not equal to $\proj_{A} (C_{i})$ for $i \in \{1, \dots, t-1\}$.

With Proposition 3.19.7 and Lemma 2.30.7 of \cite{Tits:74}, it follows that 
\[\begin{cases}
&\dist(C_{i}, E') =  \dist(\proj_{A} (C_{i}), C_{i}) + 1 = \dist(C_{i}, E) = \diam (\Delta) \text{, for } i \in \{1, \dots, t-1\} ,\\
&\dist(C_{t},E') = \dist(C_{t},E)+1.
\end{cases}
\]
That means $E'$ is opposite to each $C_{i}$ for $i \in \{1, \dots, t-1\}$ and has a strictly greater distance to $C_{t}$ than $E$. That contradicts the fact that $E$ has maximal distance to $C_{t}$ among the chambers opposite each $C_{i}$ for $i \in \{1, \dots, t-1\}$. It follows that $C_{t}$ and $E$ are opposite.
\end{proof}
 
 This proposition takes care of all situations where the field has order at least $4$. Over $\F_2$, the projectivity groups will always be determined already by  \cref{MR0}. So there remains to deal with $\F_3$. In this case, we will prove in the situations we need and more generally, that, if the simply laced spherical building is defined over the finite field $\F_q$, then we can find a simplex opposite $q+1$ given simplices of certain given types (see the next paragraphs).   
 
  \subsection{Projective spaces}
Here, $\Delta$ is a projective space over a skew field $\L$. We will show that the special projectivity groups of any irreducible residue of rank $\ell$ is isomorphic to $\PGL_{\ell+1}(\L)$. The general group always coincides with the special group, either because the type of the simplex is not self-opposite, or the type is polar closed. 

\begin{theorem}\label{caseA}
Let $\Delta$ be a building of type $\mathsf{A}_r$, defined over the skew field $\L$. Let $F$ be any simplex such that $I\setminus\typ(F)$ is connected in the Coxeter diagram (say of type $\mathsf{A}_\ell$). Then both $\Pi^+(F)$ and $\Pi(F)$ are permutation isomorphic to $\PGL_{\ell+1}(\L)$.
\end{theorem}
\begin{proof}
Applying \cref{MR0} and \cref{corgate}, it suffices to show that the stabiliser $G$ of a hyperplane $H$ of $\PG(r,\L)$ in $\PSL_{r+1}(\L)$ acts on $H$ as $\PGL_r(\L)$. Let $g$ be an arbitrary element of $\PGL_{r}(\L)$ acting on $H$. Then we can represent $g$ with respect to an arbitrarily chosen basis $B$ in $H$ with an $r\times r$ matrix $M$. We have to find a member $g^*\in \PSL_{r+1}(\L)$ inducing $g$ in $H$. We can extend $B$ to a basis $B^*$ of $\PG(r,\L)$ by adding one point $p_0\notin H$ and a suitable unit point. Let $d$ belong to the coset of the (multiplicative) commutator subgroup $C$ of $\L^\times$ given by the Dieudonn\'e determinant of $M$ (see \cite{Dieu:43}). Then the block matrix $M^*:=\begin{pmatrix} d^{-1} & 0\\ 0 & M\end{pmatrix}$ represents a member $g^*$ of $\PGL_{r+1}(\L)$ fixing $p_0$, stabilising $H$ and inducing $g$ in $H$. Moreover, by the properties of the Dieudonn\'e determinant, in particular those established in the proof of \cite[Theorem 1]{Dieu:43}, the determinant of $M^*$  is equal to the product of the coset $d^{-1}C$ and the coset $\det M$. By the definition of $d$, this product is exactly $C$, and so $g^*\in\PSL_{r+1}(\L)$. The proof is complete. 
\end{proof}

\subsection{Hyperbolic polar spaces}
We first prove some lemmas. When we consider residues of vertices of type $1$, that is, the points of the corresponding polar space, we will aim to apply \cref{triangles}. \cref{3.30} already tells us that we can find a point opposite $4$ arbitrarily given points if the underlying field has order at least $4$. To handle the case with the field $\F_3$, we recall the following slightly more general results for hyperbolic quadrics, proved in \cite{lines}.

\begin{lem}\label{hyppoints}\label{MSStriangle}
If every line of a hyperbolic quadric $Q$ of rank at least $3$ contains exactly $s+1$ points, then 
\begin{compactenum}[$(i)$] \item there exists a point non-collinear to each point of an arbitrary set $T$ of $s+1$ (distinct) points, except if these points are contained in a single line, and
\item if $Q$ has even Witt index $2d$, then there exists a maximal singular subspace opposite each member of an arbitrary set $T$ of $s+1$ (distinct) maximal singular subspaces of common type, except if these maximal singular subspaces contain a common singular subspace of codimension $2$ in each.  
\end{compactenum}
\end{lem}

\begin{nota} \label{Q0} For a hyperbolic quadric $Q$ of Witt index $r$, associated to the quadratic form $g\colon V\to \K$ with associated bilinear form  $f:V\times V\to\K$, we denote by $\PGO_{2r}(\K)$ the group of all elements of $\PGL_{2r}(V)$ preserving $f$ and $g$. The unique subgroup of index $2$ preserving each class of maximal singular subspaces will be denoted by $\PGO^\circ_{2r}(\K)$. Note that $\PGO^\circ_6(\K)$ is isomorphic to $\PSL_4(\K,2)$. \end{nota} 

A \emph{parabolic polarity} of $Q$ is the involution fixing a given parabolic subquadric $P$ of Witt index $r-1$ and interchanging each two maximal singular subspaces of $Q$ containing a common maximal singular subspace of $P$. Each parabolic polarity belongs to $\PGO_{2r}(\K)$, as in $V$, it is given by $V\to V\colon v\mapsto v-\frac{f(v,w)}{g(w)}w$, for some $w\in V$ with $g(w)\neq 0$.  

Since a parabolic subquadric of Witt index $r-1$ of a given hyperbolic quadric of Witt index $r$ defines a geometric hyperplane, one deduces that:

\begin{exm}\label{exDnBn}
The set of parabolic polarities of a given hyperbolic quadric, is geometric.
\end{exm}

The following lemma is a reformulation of Theorem~1.5.1 of \cite{Che:54}.
\begin{lem}\label{PP=PGO}
The group $\PGO_{2r}(\K)$ is generated by the parabolic polarities.
\end{lem}

\begin{theorem}\label{caseD}
Let $\Delta$ be the building (of rank $r\geq 4$) associated to a hyperbolic quadric $Q$ of Witt index $r\geq 4$ over the field $\K$. Let $F$ be a simplex of $\Delta$ such that $\Res_\Delta(F)$ is irreducible. Then $\Pi(F)$ and $\Pi^+(F)$ are given as in \emph{\cref{Drtabel}}. 
In Case~\emph{(A$*$)}, the permutation group   $\PGL_r(\K,2).2$ denotes the extension of  $\PGL_r(\K,2)$ by a symplectic polarity acting on $\PG(r-1,\K)$ (and coincides with the group generated by all symplectic polarities). A long hyphen in the table in the column of $\Pi(F)$ means that $\typ(F)$ is not self-opposite and so $\Pi(F)$ is trivially isomorphic to $\Pi^+(F)$ --- it must be read as a ``bysame'' symbol. Grey rows correspond to projectivity groups that are not necessarily the full linear groups. 
\end{theorem}

\begin{table}[tbh]
\renewcommand{\arraystretch}{1.2}
\begin{tabular}{|c|c|c|c|c|c|}\hline
Reference& $\Res_\Delta(F)$ &  $\cotyp(F)$ & $\Pi^+(F)$&$\Pi(F)$\\ \hline
(A1) & $\mathsf{A_1}$ & & $\PGL_2(\K)$ & $\PGL_2(\K)$\\ \hline
(A3) & \cellcolor[gray]{0.8}$\mathsf{A_3}$ & \cellcolor[gray]{0.8}$\{r-2,r-1,r\}$ &\cellcolor[gray]{0.8}$\PGO_6^\circ(\K)$ & \cellcolor[gray]{0.8}$\PGO_6(\K)$\\ \hline
(A) & $\mathsf{A_\ell}$, $2\leq\ell\leq r-2$&$\neq \{r-2,r-1,r\}$& $\PGL_{\ell+1}(\K)$ &$\PGL_{\ell+1}(\K).2$\\ \hline
 (A$^*$) & \cellcolor[gray]{0.8} $\mathsf{A}_{r-1}$, $r\in2\Z$ &\cellcolor[gray]{0.8}  & \cellcolor[gray]{0.8}$\PGL_r(\K,2)$ & \cellcolor[gray]{0.8}$\PGL_r(\K,2).2$\\ \hline
 (A$^{**}$) &  $\mathsf{A}_{r-1}$, $r\in2\Z+1$ &  & $\PGL_r(\K)$ &  ------\\ \hline
(D)& \cellcolor[gray]{0.8}$\mathsf{D}_{r-2\ell}$,  $4\leq r-2\ell\leq r-1$ & \cellcolor[gray]{0.8} &\cellcolor[gray]{0.8} $\PGO^\circ_{2r-4\ell}(\K)$ &\cellcolor[gray]{0.8}$\PGO^\circ_{2r-4\ell}(\K)$  \\ \hline
(D$'$)&\cellcolor[gray]{0.8}$\mathsf{D}_{r-2\ell+1}$,  $4\leq r-2\ell+1\leq r-1$ & \cellcolor[gray]{0.8} & \cellcolor[gray]{0.8}$\PGO^\circ_{2r-4\ell+2}(\K)$ &\cellcolor[gray]{0.8}$\PGO_{2r-4\ell+2}(\K)$  \\ \hline
\end{tabular}\vspace{.3cm}
\caption{Projectivity groups in buildings of type $\mathsf{D}_r$ over $\K$\label{Drtabel}}
\end{table}

\begin{proof}First we notice that, if $\K=\F_2$, then all groups are universal and adjoint (simple) at the same time, so the results follow from \cref{MR0}. Hence we may assume $|\K|\geq 3$.  For ease of notation and language, we will speak about plus-type and minus-type of the maximal singular subspaces of $Q$ to distinguish the two different types (arbitrarily). 

Also, Case (A1) follows from \cref{MR3}, whereas Case (A) follows from \cref{gate2} and \cref{caseA}. We now handle the other, less straightforward, cases.
\medskip

\framebox{Case (A$^*$)} Let $M_1,M_2,M_3$ be three mutual opposite maximal singular subspaces of plus-type. Let $p_1\in M_1$ be arbitrary. The maximal singular subspace $N$ through $p_1$ intersecting $M_2$ in a submaximal singular subspace (that is, a singular subspace of dimension $r-2$) intersects $M_3$ in a point $p_3$, since $N$ is necessarily of minus-type. Hence the maximal singular subspace of minus-type through $p_3$ intersecting $M_1$ in a submaximal singular subspace contains $p_1$. This shows that the projectivity $M_1\per M_2\per M_3\per M_1$ is a duality each point of which is absolute. Lemma~3.2 of \cite{Tem-Tha-Mal:11} implies that it is a symplectic polarity. By conjugation, we can obtain every symplectic polarity of $M_1$ in this way. Applying \cref{triangles} together with \cref{MSStriangle}, Case (A$^*$) follows from \cref{exAnCn} and the fact that the matrix corresponding to a symplectic polarity necessarily has square determinant (and every square can occur). 
\medskip

\framebox{Case (A3)} By \cref{MR0}, every self-projectivity preserves the residual form, hence $\Pi(F)\leq \PGO_6(\K)$. Case (A$^*$) for $r=3$, together with \cref{gate2} and the fact that  $\PGO^\circ_6(\K)$ is isomorphic to $\PSL_4(\K,2)$, conclude this case.
 
 \framebox{Case (A$^{**})$} By \cref{MR0}, $\Pi^+(F)$ contains $\PSL_r(\K)$. Hence it suffices to show that $\Pi^+(F)$ contains an element of $\PGL_r(\K)$ whose corresponding matrix has arbitrary determinant. 
 
Let $M_1$ and $M_3$ be two maximal singular subspaces  of plus-type intersecting in a subspace $U_{13}$ of dimension $r-3$. Let $M_2$ be a maximal singular subspace opposite both $M_1$ and $M_3$ (then $M_2$ has plus-type). Let $U_{24}$ be a subspace of $M_2$ of dimension $r-3$ opposite $U_{13}$. Let $L_{1}$ be the unique line of $M_1$ collinear to $U_{24}$. Let $L$ be an arbitrary line in $M_1$ joining a point $p_{13}\in U_{13}$ with some point $p_{1}\in L_{1}$. Pick $p,p'\in L\setminus\{p_{13},p_1\}$ and suppose $p\neq p'$ (this is possible as we assume $|\K|\geq 3$). 

Let $M$ be the maximal singular subspace of plus-type containing $p$ and intersecting $M_2$ in a hyperplane. Denote $W=U_{24}\cap M$. Then $W$ has dimension $r-4$ and is collinear to $L$. The intersection of $M$ and $M_3$ is a point $q$, as both have the same type. As both $p$ and $p_{13}$ are collinear to $q$, also $p'$ is collinear to $q$. Hence $p'$ is collinear to $\<q,W\>$, and $\<p',q,W\>$ is a singular subspace of dimension $r-2$. Hence there is a unique maximal singular subspace $M'$ of plus-type containing $p',q$ and $W$. It obviously intersects $M_1$ in $p'$ and $M_3$ in $q$. There is a unique maximal singular subspace $M_4$ containing $U_{24}$ and intersecting $M'$ in a hyperplane (and hence it is of minus-type). Now with this set-up, one verifies that the projectivity $M_1\per M_2\per M_3\per M_4\per M_1$ pointwise fixes both $U_{13}$ and $L_1$, and maps $p$ to $p'$. Choosing a basis in $U_{13}\cup L_1$, the matrix of such a homology in $M_1$  is diagonal of the form $\diag(k,k,\ell,\ell,\ldots,\ell)$, and the arbitrariness of $p'$ implies that $k$ and $\ell$ are also arbitrary. Set $r=2s+1$. Putting $k=\ell^{-s+1}$, we obtain the determinant $\ell^{-2s+2+2s-1}=\ell$. Since $\ell$ is arbitrary, the assertion follows.       

\framebox{Case (D$'$)} First set $\ell=1$, that is, $r-2\ell+1=r-1$ and $F$ is just a point of the polar space or hyperbolic quadric $Q$. Let $p_1,p_2,p_3$ be three mutual opposite points. Since $p_1^\perp\cap p_2^\perp$ is a hyperbolic quadric of rank $r-1$, we have that $p_1^\perp\cap p_2^\perp\cap p^\perp_3$ is either a parabolic subquadric, or a degenerate quadric. In the latter case, $\{p_1,p_2,p_3\}^{\perp\perp}$ is a degenerate plane conic containing $p_1,p_2,p_3$, and hence $p_3$ is collinear to either $p_1$ or $p_2$, a contradiction. Consequently $p_1^\perp\cap p^\perp_2\cap p^\perp_3$ is a parabolic quadric and the projectivity $p_1\per p_2\per p_3\per p_1$ is a parabolic polarity. Clearly, every parabolic polarity of $\Res_\Delta(p_1)$ can be obtained this way. Then \cref{triangles}, \cref{hyppoints}, \cref{exDnBn} and \cref{PP=PGO} yield $\Pi(p_1)=\PGO_{2r-2}(\K)$ and $\Pi^+(p_1)=\PGO^\circ_{2r-2}(\K)$.  

Now let $\ell$ be arbitrary (but of course $4\leq r-2\ell+1\leq r-1$). Since the stabiliser of $F$ in $\PGO_{2r}(\K)$ obviously preserves the residual form (in $\Res_\Delta(F)$), we see that $\Pi^+(F)$ is a subgroup of $\PGO^\circ_{2r-4\ell+2}(\K)$, and hence coincides with it by \cref{gate2} and the case $\ell=1$. In order to show $\Pi(\F)=\PGO_{2r-4\ell+2}(\K)$, we only need to exhibit a parabolic polarity as a self-projectivity in $\Res_\Delta(F)$. This is done similarly as in the previous paragraph for the case $\ell=1$: choose three mutual opposite singular subspaces $U_1,U_2,U_3$ of dimension $2\ell-1$ contained in a parabolic subquadric obtained from $Q$ by intersecting $Q$ in its ambient projective space with a subspace of dimension $4\ell$. Suppose also $U_1\in F$. Then, as before, the projectivity $U_1\per U_2\per U_3\per U_1$ is a parabolic polarity of $\Res_\Delta(F)$. 

\framebox{Case (D)} This is completely similar to the case $\ell>1$ of Case (D$'$), noting that $\Pi^+(F)$ coincides with $\Pi(F)$ by \cref{MR1}.
\end{proof}

\subsection{Exceptional cases}
Also here, we first prove some lemmas. First we recall the following result from \cite{lines2} in order to deal with the case of a field of order $3$ for simplices of type $7$ in $\mathsf{E_7}$. 
\begin{lem}\label{E77points}
If every line of a parapolar space $\Gamma$ of type $\mathsf{E_{7,7}}$ contains exactly $s+1$ points, then there exists a point at distance $3$ from each point of an arbitrary set $T$ of $s+1$ (distinct) points, except if these points are contained in a single line. 
\end{lem}

\begin{nota}[Similitudes---Groups of type $\mathsf{D}_n$]\label{Q}
For a hyperbolic quadric $Q$ of Witt index $r$, associated to the quadratic form $g\colon V\to \K$ with associated bilinear form  $f\colon V\times V\to\K$, we denote by  $\overline{\mathsf{PGO}}_{2r}(\K)$ the group of all elements of $\PGL_{2r}(V)$ preserving $f$ and $g$ up to a scalar multiple. It is the complete linear (algebraic) group of automorphisms of $Q$, seen as a building of type $\mathsf{D}_r$. The unique subgroup of  $\overline{\mathsf{PGO}}_{2r}(\K)$ of index $2$ preserving each class of maximal singular subspaces will be denoted by $\overline{\mathsf{PGO}}^\circ_{2r}(\K)$. It is elementary to see that  $\overline{\mathsf{PGO}}_{2r}(\K)$  is obtained from  ${\mathsf{PGO}}_{2r}(\K)$ by adjoining the appropriate \emph{diagonal automorphisms}, that is, if we assume $g$ in standard form (after introducing coordinates) 

\begin{align*} g\colon \K^{2r}\to \K\colon & (x_{-r},x_{-r+1},\ldots,x_{-2},x_{-1},x_1,x_2,\ldots,x_{r-1},x_r)\\ &\mapsto x_{-r}x_r+x_{-r+1}x_{r-1}+\cdots x_{-2}x_2+x_{-1}x_1,\end{align*} 
then we adjoin the linear automorphisms of $Q$ induced by 

 \begin{align*}\varphi_k\colon\K^{2r}\to\K^{2r}\colon & (x_{-r},x_{-r+1},\ldots,x_{-2},x_{-1},x_1,x_2,\ldots,x_{r-1},x_r)\\& \mapsto  (x_{-r},x_{-r+1},\ldots,x_{-2},x_{-1},kx_1,kx_2,\ldots,kx_{r-1},kx_r),\end{align*} for all $k\in\K^\times$ (and we may assume $k$ is not a square as otherwise the given automorphism is already in ${\mathsf{PGO}}_{2r}(\K)$). 
We denote the commutator subgroup of ${\mathsf{PGO}}^\circ_{2r}(\K)$ by $\mathsf{P\Omega}_{2r}(\K)$. The latter is the simple group $\mathsf{D}_{r}(\K)$ of type $\mathsf{D}_r$ over the field $\K$ (see \cite{Dieu:62}). The group obtained from  $\mathsf{P\Omega}_{2r}(\K)$ by adjoining the diagonal automorphisms as above is denoted by $\overline{\mathsf{P\Omega}}_{2r}(\K)$. 
\end{nota}

If $r$ is even and $\K$ is not quadratically closed, then  $\overline{\mathsf{P\Omega}}_{2r}(\K)$ does not coincide with  $\overline{\mathsf{PGO}}^\circ_{2r}(\K)$ as we will demonstrate later (see \cref{omega}).

Let us call \emph{homology} of a hyperbolic quadric $Q$ as in \cref{Q} any automorphism of $Q$ pointwise fixing two opposite maximal singular subspaces. The automorphisms $\varphi_k$, $k\in\K^\times$, above are homologies. If $r$ is even, then there are two types of such according to which kind of maximal singular subspaces is fixed pointwise (if $r$ is odd, then one always pointwise fixes one maximal singular subspace of each type). We now have the following result, which can be proved using standard arguments similarly to, but simpler than, \cref{PP=PGO} and \cref{polE6}.

\begin{lem}\label{PGO}
Let $Q$ be a (non-degenerate) hyperbolic quadric of Witt index $r$ corresponding to the building of type $\mathsf{D}_r$ over the field $\K$. Then the following hold.
\begin{compactenum}[$(i)$]
\item The set of all homologies generates $\overline{\PGO}_{2r}^\circ(\K)$.
\item If $r$ is even, then the set of homologies pointwise fixing two opposite maximal singular subspaces of only one given type generates $\overline{\mathsf{P\Omega}}_{2r}(\K)$. 
\item If $r$ is even, then the homologies pointwise fixing two opposite maximal singular subspaces of only one given type, and the elements of $\PGO_{2r}^\circ(\K)$ together generate $\overline{\PGO}_{2r}^\circ(\K)$.
\item The set of all automorphisms fixing two opposite points $p,q$ and pointwise fixing $p^\perp\cap q^\perp$ generates $\PGO_{2r}^\circ(\K)$.
\end{compactenum}
\end{lem}

Let $U$ and $U'$ be two opposite maximal singular subspaces of a hyperbolic quadric. Then it is well known that every point not contained in either $U$ or $U'$ is contained in a unique line joining a point of $U$ and a point of $U'$. One can use this property to deduces that if a non-trivial homology pointwise fixing $U\cap U'$ fixes a subspace $S$, then either $S\cap U$ and $S\cap U'$ generate $S$. This, in turn, implies:

\begin{exm}\label{exDn}
The set of all non-trivial homologies of a hyperbolic quadric pointwise fixing two opposite maximal singular subspaces of a given type, is geometric. 
\end{exm}

We now introduce some notation concerning the exceptional groups of type $\mathsf{E_6}$. There does not seem to be standard notation (some people use $\tilde{E}$ and $\hat{E}$, others SE$_6(\K)$ for some of the following groups). The following is partly based on \cite{Ste:20}.
\begin{nota}[Groups of type $\mathsf{E_6}$]
Let $V$ be a $27$-dimensional vector space over the commutative field $\K$, written as the direct sum of three $1$-dimensional subspaces and three $8$-dimensional subspaces, each of them identified with a split octonion algebra $\O$ over $\K$. We thus write $V=\K\oplus\K\oplus\K\oplus\O\oplus\O\oplus\O$. Let $\mathfrak{C}:V\rightarrow\K$ be the cubic form defined as
\[\mathfrak{C}(x,y,z;X,Y,Z)=-xyz+xX\overline{X}+yY\overline{Y}+zZ\overline{Z}-(XY)Z-\overline{(XY)Z}.\]
Then we denote by $\GE_6(\K)$ the similitudes of $\mathfrak{C}$, that is, the subgroup of $\mathsf{GL}(V)$ preserving $\mathfrak{C}$ up to a multiplicative constant. The subgroup of $\GE_6(\K)$ preserving $\mathfrak{C}$ is denoted by $\SE_6(\K)$ (and is a subgroup of $\SL(V)$) and the quotients with the  respective centres (consisting of scalar matrices) are $\PGE_6(\K)$ and $\mathsf{PSE}_6(\K)$. The latter is also denoted briefly by $\mathsf{E}_6(\K)$ and is simple. The group $\mathsf{PGE_6}(\K)$ is the full linear group. The group obtained by adjoining a graph automorphism is denoted by $\PGE_6(\K).2$.
\end{nota}

The cubic form $\mathfrak{C}$ above can also be written without the use of octonions, but using the unique generalised quadrangle $\GQ(2,4)$ of order $(2,4)$, that is, polar space of rank $2$ with $3$ points on each line and $5$ lines through each points.  An explicit construction of $\GQ(2,4)$ runs as follows, see Section~6.1 of \cite{FGQ}. Let $\cP'$ be the set of all $2$-subsets of the $6$-set $\{1,2,3,4,5,6,\}$, and define \[\cP=\cP'\cup\{1,2,3,4,5,6\}\cup\{1',2',3',4',5',6'\}.\] Denote briefly the $2$-subset $\{i,j\}$ by $ij$, for all appropriate $i,j$. Let $\cL'$ be the set of partitions of $\{1,2,3,4,5,6\}$ into $2$-subsets and define \[\cL=\cL'\cup\left\{\{i,j',ij\}\mid i,j\in\{1,2,3,4,5,6\}, i\neq j\right\}.\] Then $\Gamma=(\cP,\cL)$ is a model of $\GQ(2,4)$. 

The sets $\{1,2,3,4,5,6\}$ and $\{1',2',3',4',5',6'\}$ have the property that they both do not contain any pair of collinear points, and that non-collinearity is a paring between the two sets. Such a pair of $6$-sets is usually called a \emph{double six}. 

 The following set $\cS$ of lines of $\GQ(2,4)$ is a  \emph{spread}, that is, a partition of the point set $\cP$ into lines: $\cS=$
%\begin{eqnarray*}\cS&=&\left\{\{14,25,36\},\{15,26,34\},\{16,24,35\},\right.\\ &&\{12,2,1'\},\{23,3,2'\},\{13,1,3'\},\\&&\left.\{45,4,5'\},\{56,5,6'\},\{46,6,4'\}\right\}.\end{eqnarray*} 
\[\left\{\{14,25,36\},\!\!\{15,26,34\},\!\!\{16,24,35\},\!\!\{12,2,1'\},\!\!\{23,3,2'\},\!\!\{13,1,3'\},\!\!\{45,4,5'\},\!\!\{56,5,6'\},\!\!\{46,6,4'\}\right\}.\]
We now have the following equivalent description of the cubic form $\mathfrak{C}$, see  Section~$2$ of \cite{Mal-Vic:22}. Let $V$ be the vector space of dimension 27 over $\K$ where the standard basis $B$ is labelled using the elements of $\cP$, say $B=\{e_p\mid p\in\cP\}$. We denote a generic vector $v\in V$ by $\sum_{p\in\cP}x_pe_p$, with $x_p\in\K$. Then  \[\mathfrak{C}(v)=\mathlarger{\mathlarger{\sum}}_{\{p,q,r\}\in\cS} x_px_{q}x_{r}-\mathlarger{\mathlarger{\sum}}_{\{p,q,r\}\in\cL\setminus\cS} x_px_{q}x_{r}.\]

The projective null set of $\nabla\mathfrak{C}$ is a set of points denoted $\cE_6(\K)$, and endowed with the lines contained in it, it is a point-line geometry isomorphic to the Lie incidence geometry of type $\mathsf{E_{6,1}}$ over the field $\K$. There is a special type of graph automorphism, called \emph{symplectic polarity}, which is an involution centralising a split group of type $\mathsf{F_4}$ over the same field $\K$. All symplectic polarities are conjugate (see \cite{Sch-Sas-Mal:18}), and as a consequence of the main results of \cite{Mal:12}, every graph automorphism having a fix set isomorphic to the fix set of a symplectic polarity, is a symplectic polarity. Hence:

\begin{exm}\label{exE6}
The set of symplectic polarities of a building of type $\mathsf{E_6}$ is geometric. 
\end{exm}

With these constructions and notation at hand, we are able to prove the following generation results. 

\begin{lem}\label{polE6}
The products of an even number of symplectic polarities of the building $\Delta$ of type $\mathsf{E_6}$ over the field $\K$ generate the Chevalley group $\PGE_6(\K)$.  The symplectic polarities themselves generate $\PGE_6(\K).2$.
\end{lem}
\begin{proof}
We first claim that diagonal automorphisms $\varphi$ necessarily have a ninth power as determinant.  Indeed, $\varphi$ acts as \[\varphi\colon V\to V\colon  \sum_{p\in\cP}x_pe_p\mapsto \sum_{p\in\cP}\lambda_px_pe_p,\] with $\lambda_p\in\K$. Then $\varphi$ is a similitude of $\mathfrak{C}$ only if the product $\lambda_p\lambda_q\lambda_r=:\lambda$ is a constant across all lines $\{p,q,r\}\in\cL$. Then the determinant of $\varphi$ is obtained by multiplying this constant over the spread $\cS$ and hence the determinant equals $\lambda^9$. The claim is proved. 

Now a symplectic polarity $\sigma$ of $\Delta$ induces a symplectic polarity in every fixed $5$-space of the corresponding Lie incidence geometry $\Gamma$ of type $\mathsf{E_6}$. Since all symplectic polarities are conjugate, every symplectic polarity of a given $5$-space extends to a symplectic polarity of $\Delta$. By the strong transitivity of $\Aut\Delta$, we may even assume that two given opposite $5$-spaces are fixed (and then, since the fix building has type $\mathsf{F_4}$ and its polar type corresponds to the fixed $5$-spaces (as is apparent from \cite{Sch-Sas-Mal:18}, all $5$-spaces in the span of the two given ones in $\PG(V)$ are fixed). Now, two opposite $5$-spaces $W$ and $W'$ are given by the span of the base points corresponding to the two respective $6$-sets of a double six. It is easily seen that the product of the symplectic polarities corresponding to the symplectic forms \[x_{-3}y_3+x_{-2}y_2+x_{-1}y_1-x_1y_{-1}x_2y_{-2}-x_3y_{-3}\mbox{ and }\lambda x_{-3}y_3+x_{-2}y_2+x_{-1}y_1-x_1y_{-1}x_2y_{-2}-\lambda x_3y_{-3},\] $\lambda\in\K$, corresponds to the diagonal collineation of $\PG(5,\K)$ with diagonal $(\lambda,1,1,1,1,\lambda)$. Letting these coordinates correspond naturally to the bases $(e_1,\ldots,e_6)$ and $(e_{1'},\ldots,e_{6'})$ of the subspaces of $V$ corresponding to $W$ and $W'$, respectively, we first derive that the product $\theta$ of the corresponding symplectic polarities of $\Delta$ acts on $\<W,W'\>$ as \[(x_1,x_2,\ldots,x_6,x_{1'},x_{2'},\ldots,x_{6'})\mapsto (\lambda x_1,x_2,x_3,x_4,x_5,\lambda x_6,\lambda x_{1'},x_{2'},x_{3'},x_{4'},x_{5'},\lambda x_{6'}).\] Secondly, since each point $\<e_{ij}\>$ is the unique point of $\Gamma$ collinear to all $\<e_\ell\>$, except for $\<e_i\>$ and $\<e_j\>$, and to all  $\<e_{\ell'}\>$, except for $\<e_{i'}\>$ and $\<e_{j'}\>$ (as follows from Lemma~3.5 in \cite{ADS-HVM}), we see that $\theta$ is a diagonal automorphism. Now one easily calculates that $\theta$ is uniquely determined by its restriction to $\<W,W'\>$ and maps $e_{ij}$ to $e_{ij}$ if $|\{i,j\}\cap\{1,6\}|=1$, to $\lambda e_{ij}$ if $\{i,j\}\cap\{1,6\}=\varnothing$, and to $\lambda^{-1}e_{ij}$ if $\{i,j\}=\{1,6\}$. Correspondingly, the determinant of $\theta$ is $\lambda^9$.  Now clearly the diagonal automorphisms generate non-trivial elements of $\mathsf{PSE_6}(\K)$. Since the latter is simple, and since the subgroup of $\mathsf{E_6}(\K).2$ generated by all symplectic polarities of $\Delta$ is normal, the group generated by arbitrary products of an even number of symplectic polarities contains $\mathsf{PSE_6}(\K)$.  Since it also contains all diagonal automorphisms by the above, the assertions follow. 
\end{proof}

 \begin{lem}\label{E6T1}
With the notation of \emph{\cref{algapp}}, we have $\mathsf{E_6}(\K).T_1=\mathsf{PGE_6}(\K)$.
\end{lem}

\begin{proof}
A chamber $C$ of a building of type $\mathsf{E_6}$ over $\K$ is given by a $6$-tuple of pairwise incident elements of the Lie incidence geometry of type $\mathsf{E_{6,1}}$ over $\K$ described above consisting of a point, a  line, a plane,  a $5$-space, a $4$-space intersecting the $5$-space in a $3$-space, and a hyperplane of the $5$-space.  In the above description, we can take, with obvious notation, \[C=(\<e_1\>,\<e_1,e_2\>,\<e_1,e_2,e_3\>,\<e_1,e_2,\ldots,e_6\>,\<e,1,\ldots,e_4,e_{\{5,6\}}\>,\<e_1,\ldots,e_5\>).\]
Denote $W=\<e_1,\ldots,e_6,e_{\{5,6\}}\>$. An arbitrary diagonal automorphism of $W$ inducing an element of $T_1$, that is, acting trivially on the rank 1 residues defined by $C$, except for the type $1$ rank 1 residue, is given by \[ \begin{cases}e_1\mapsto \lambda e_1, &\\ e_i\mapsto e_i,& i=2,3,4,5,6,\\ e_{\{5,6\}}\mapsto e_{\{5,6\}}.\end{cases} \] It is easy to check that this can be extended to a unique diagonal automorphism preserving the cubic form $\mathfrak{C}$ by defining
 \[\begin{cases}
 e_{1'}\mapsto e_{1'},&\\
 e_{i'}\mapsto \lambda^{-1}e_{i'},& i=2,3,4,5,6,\\
 e_{\{i,j\}}\mapsto e_{\{i,j\}}, & 1\notin\{i,j\}\subseteq\{1,2\ldots,6\}, i\neq j,\\
 e_{\{1,j\}}\mapsto \lambda^{-1}e_{\{1,j\}}, & j=2,3,4,5,6.
 \end{cases}\]
Now one calculates that the determinant of the corresponding diagonal matrix is $\lambda^{-9}$, and the result follows similarly as in the proof of \cref{polE6} above.
\end{proof}

\cref{triangles} and \cref{E77points} can be used to determine $\Pi^+(F)$ and $\Pi(F)$ for simplices of type $7$ in buildings of type $\mathsf{E_7}$. However, in general, the number of possibilities for triangles of mutually opposite simplices is too large to be practical or useful. The following result provides an alternative to \cref{triangles}. The condition that $J$ is a self-opposite type is not essential, but convenient, and we will only need it in that case.  

\begin{nota}
For a spherical building $\Delta$ of type $\mathsf{X}_n$ over $I$ and a type set $J\subseteq I$, we denote by $\Gamma_J$ the graph with vertices the simplices of type $J$, adjacent when contained in adjacent chambers. Adjacent vertices $F,F'$ in $\Gamma_J$ are denoted $F\sim F'$. An alternative definition of $\Gamma_J$ is that it is the point graph of the corresponding Lie incidence geometry of type $\mathsf{X}_{n,J}$.
\end{nota}

 \begin{lem}\label{upanddown}
 Let $\Delta$ be a spherical building over the type set $I$ and let $J\subseteq I$ be a self-opposite type. Suppose that for each pair of simplices $F,F'$ of type $J$, the subgraph $\Gamma_{J}^{\{F,F'\}}$ of $\Gamma_J$ induced on the vertices opposite both $F$ and $F'$ is connected. Suppose also that there is a simplex of type $J$ opposite any given set of three simplices of type $J$. Let $F$ be a given simplex of type $J$. Denote by $\Pi_4(F)$ the set of all self-projectivities $F\per F_2\per F_3\per F_4\per F$ of $F$ of length $4$ with $F\sim F_3$, $F_2\sim F_4$. Suppose that $\Pi_4(F)$ is geometric. Then $\Pi^+(F)=\<\Pi_4(F)\>$. 
 \end{lem}
\begin{proof}
We first prove the following property for four simplices $F_1,F_2,F_3,F_4$, where $\typ(F_1)=\typ(F_3)=J$ and both $F_2$ and $F_4$ are opposite both $F_1$ and $F_3$.
\begin{itemize}
\item[(*)] \emph{The projectivity $\rho\colon F_1\per F_2\per F_3\per F_4$ can be written as a product of a perspectivity $F_{1} \per F_{4}$ and conjugates of members of $\Pi_4(F_1)$.}
\end{itemize}
Indeed, let $F_1= F_1'\sim F_2'\sim\cdots F_n'=F_3$ be a  path in $\Gamma_{J}^{\{F_{2},F_{4}\}}$.  Define $\rho_i\colon F_4\per F_i'\per F_2\per F_{i+1}'\per F_4$, $i \in\{1,2,\ldots, n-1\}$. Denote by $\rho_0$ the perspectivity $F_1\per F_4$. Then it is elementary to see that $\rho=\rho_0\rho_1\rho_2\cdots\rho_{n-1}$. So, since $\Pi_4(F)$ is geometric, it suffices to show that each $\rho_i$ can be written as the product of conjugates of members of $\Pi_4(F_1)$. It follows from letting $(F_4,F_i',F_2,F_{i+1}')$ play the role of $(F_1,F_2,F_3,F_4)$ in the previous argument that $\rho_{i}$ is a product of conjugates of members of $\Pi_4(F_{i+1}')$. Hence (*) is proved.

Now let $\rho\colon F\per F_2\per F_3\per \cdots\per F_{2\ell-1}\per F_{2\ell}\per F$ be an arbitrary even projectivity. We prove by induction on $\ell\in\{1,2,\ldots\}$ that $\rho$ is the product of conjugates of members of $\Pi_4(F)$. This is trivial for $\ell=1$ and it is equivalent to (*) for $\ell=2$. So let $\ell\geq 3$. Select a simplex $F_2'$ opposite each of $F, F_3$ and $F_5$ (since these all have the same type, this is still possible if $J$ is not self-opposite).  Setting \vspace{-.5cm} \[\begin{cases}\rho^*_1\colon F\per F_2\per F_3\per F_2'\per F,\\ \rho^*_2\colon F_2'\per F_3\per F_4\per F_5\per F_2',\\ \rho'\colon F\per F_2'\per F_5\per F_6\cdots\per F_{2\ell}\per F,\end{cases}\] we see that, if $\rho_0\colon F\per F_2'$, we have $\rho=\rho_1^*\cdot(\rho_0\rho_2^*\rho_0^{-1})\cdot\rho'$, where we know by the induction hypothesis that all factors are products of members of $\Pi_4(F)$, using the fact that $\Pi_4(F)$ is geometric (and hence closed under conjugation).  
\end{proof}

This method of determining $\Pi^+(F)$ ``explains'' these groups in a way  equivalent to the algebraic approach of \cref{algapp}.  

In order to be able to apply \cref{upanddown}, we have to check the conditions in the various cases. It turns out we will use \cref{upanddown} in exactly three different cases, for which we now note down the condition on the corresponding graph $\Gamma_J$.

\begin{lem}\label{conn}
Let $\Delta$ be the spherical building over the field $\K$, $|\K|>2$, of type either $\mathsf{E_6}$ or $\mathsf{E_7}$. Let $J=\{2\}$ if $\typ\Delta=\mathsf{E_6}$, and $J\in\{\{1\},\{3\}\}$ if $\typ\Delta=\mathsf{E_7}$. Let $v$ and $v'$ be vertices of type $J$. Then the subgraph $\Gamma_J^{v,v'}$ of $\Gamma_J$ induced on the vertices opposite both $v$ and $v'$ is connected. %If $J\neq\{3\}$, then, more generally, the complement of the union of two geometric hyperplanes of the corresponding geometries of type $\mathsf{E}_{n,J}$, $n=6,7$, is connected. 
\end{lem}

\begin{proof}
Select chambers $C$ and $C'$ of $\Delta$ containing $v$ and $v'$, respectively. Let $u$ be a vertex of $\Gamma_J^{v,v'}$. In $\Res_\Delta(u)$, we can find a chamber opposite the projections of $C$ and $C'$ onto $\Res_\Delta(u)$ (use \cite[Proposition~3.30]{Tits:74}). That chamber is, by \cref{Tits}, opposite both $C$ and $C'$. Hence we have found a chamber $C_u$ containing $u$ opposite both $C$ and $C'$. If $u'$ is another vertex of $\Gamma_J^{v,v'}$, then we can also find a chamber $C_{u'}$ containing $u'$ and opposite both $C$ and $C'$. Now \cref{connopp} implies that we can find a sequence of consecutively adjacent chambers, all opposite both $C$ and $C'$, connecting $C_u$ with $C_{u'}$. the vertices of type $J$ of two consecutive such chambers are either equal, or adjacent in $\Gamma_J$. Moreover, since the chambers are opposite both $C$ and $C'$, their vertices of type $J$ are opposite both $v$ and $v'$, as $J$ is a self-opposite type. 

 The lemma is proved.
 \end{proof}

\begin{remark}It is clear that the previous lemma holds for all spherical buildings of simply laced type and self-opposite subset $J$ of the types, with the same proof. If $J$ is not self-opposite, then one has to consider the subgraph induced on the set of simplices of type $J$ opposite two given simplices of opposite type of $J$. 
\end{remark}

Before we can determine in a geometric way the various projectivity groups in the exceptional buildings of simply laced type, we need some basic properties of Lie incidence geometries of types $\mathsf{E_{6,1}}$ and $\mathsf{E_{7,7}}$. Most of them can be read off the diagram, and others follow from considering an apartment of the building. They are called ``facts'' in papers like \cite{ADS-JS-HVM,ADS-HVM}. For Lie incidence geometries of type $\mathsf{E_{6,1}}$, good references are \cite{Sch-Sas-Mal:18} and \cite{Tits:57}, and for Lie incidence geometries of type $\mathsf{E_{7,7}}$ a good reference is \cite{Sch-Sas-Mal:22}. In both papers, the basic facts are explained in some more detail. %On the other hand, Lie incidence geometries of type $\mathsf{E_{8,8}}$ are so-called \emph{long root subgroup geometries} and therefore satisfy some well-known generic properties that can be found in \cite{Coh-Iva:07}.

\subsection*{ Lie incidence geometries of type $\mathsf{E_{6,1}}$} These geometries have diameter $2$ and contain no special pairs of points. Hence, every pair of points is contained in a symp. Symps are polar spaces of type $\mathsf{D_5}$. The basic properties, which we shall use without notice, are summarised in the following lemma.
 \label{propE6E7}
\begin{lem}
Let $\Gamma$ be a Lie incidence geometry of type $\mathsf{E_{6,1}}$ over a field $\K$. Then the following properties hold. 
\begin{compactenum}[$(i)$]
\item Two distinct symps either meet in a unique point, or share a maximal singular subspace, referred to as a $4$-space. 
\item For a point $p$ and a symp $\xi$, with $p\notin\xi$, we either have $p^\perp\cap\xi=\varnothing$, or $p^\perp\cap\xi$ is a maximal singular subspace of $\xi$, referred to as a $4'$-space. 
\item The $4$-spaces in a given symp form one natural class of maximal singular subspaces of $\xi$; the $4'$-spaces form the other. 
\end{compactenum}
\end{lem}
In the building, $4$-spaces correspond to vertices of type $5$, whereas $4'$-spaces correspond to simplices of type $\{2,6\}$. 

We now mention some other facts. The first one can be read off the diagram. It is also contained as Fact~4.14 in \cite{ADS-JS-HVM}.

\begin{lem}\label{44'5}
Let $\Gamma$ be a Lie incidence geometry of type $\mathsf{E_{6,1}}$ over a field $\K$. Then the following hold. 
\begin{compactenum}[$(i)$] 
\item A $4$-space and a $4'$-space, that have a plane $\pi$ in common, intersect in a $3$-space. Consequently, a $4$-space and a $5$-space, that share a plane, share a $3$-space. 
\item Two distinct non-disjoint $5$-spaces intersect in either a point or a plane. Consequently, a $4'$-space, that shares a $3$-space with a $5$-space, is contained in it.
\item Two disjoint $5$-spaces, that are not opposite, contain respective planes contained in a common $5$-space. Every point of each of the two $5$-spaces is collinear to some point of the plane contained in the other $5$-space.  
\end{compactenum}
 \end{lem}
 
\begin{lem}\label{plane-symp}
Let $\Gamma$ be a Lie incidence geometry of type $\mathsf{E_{6,1}}$ over a field $\K$. Let $\xi$ be a symp in $\Gamma$ and let $\pi$ be a plane in $\Gamma$ intersecting $\xi$ in a unique point $x$. Then there exists a unique plane $\alpha\subseteq\xi$, all points of which are collinear to all points of $\pi$.  
\end{lem}
\begin{proof}
Let $L$ be a line in $\pi$ not intersecting $\xi$. The lemma follows from Fact A.9 of \cite{ADS-HVM}.
\end{proof}

\subsection*{ Lie incidence geometries of type $\mathsf{E_{7,7}}$} These geometries have diameter $3$ and contain no special pairs of points. Points at distance $3$ correspond to opposite vertices of type $7$ in the corresponding building. Hence every pair of non-opposite points is contained in a symp. Symps are polar spaces of type $\mathsf{D_6}$. The basic properties, which we shall use without notice, are summarised in the following lemma.

\begin{lem}\label{point-sympE7}
Let $\Gamma$ be a Lie incidence geometry of type $\mathsf{E_{7,7}}$ over a field $\K$. Let $x$ be a point and $\xi$ a symp. Then either
\begin{compactenum}[$(i)$]
\item $x\in\xi$, or
\item $x\notin\xi$, $x$ is collinear to each point of a unique $5'$-space of $\xi$ and symplectic to all other points of $\xi$, or
\item $x\notin\xi$, $x$ is collinear to a unique point $x'$ of $\xi$, symplectic to all points of $\xi$ collinear to $x'$, and opposite each other point of $\xi$.
\end{compactenum}
\end{lem}
In Case $(ii)$ above, the point $x$ is said \emph{to be close to $\xi$}, whereas in Case $(iii)$ it is said \emph{to be far from $\xi$}. 

Two distinct symps sharing at least a plane, share a $5$-space. Again, the $5$-spaces in a given symp form one natural class of maximal singular subspaces, whereas the $5'$-spaces form the other class. 

\begin{lem}\label{symppencilE7}
Let $\Gamma$ be a Lie incidence geometry of type $\mathsf{E_{7,7}}$ over a field $\K$. Let $M$ be a maximal $5$-space and let $\xi$ and $\xi'$ be two distinct symps containing $M$. Let $p\in \xi\setminus M$ and $p'\in p^\perp\cap (\xi'\setminus M)$. Then every point on the line $\<p,p'\>$ is contained in a (unique) symp, which contains $M$.  
\end{lem}

\begin{proof}
Since $p'\perp p$, we have $p^\perp\cap M=p'^\perp\cap M=:U$ is a $4$-space. Then, every point $q\in pp'$ is contained in the $5'$-space generated by $q$ and $U$, which is itself contained in a unique symp by definition of $5'$-space. 
\end{proof}

We will also need the following two results, which follow from considering an appropriate apartment.

\begin{lem}\label{sympline}
Let $\Gamma$ be a Lie incidence geometry of type $\mathsf{E_{7,7}}$ over a field $\K$. Let $\xi$ and $\xi'$ be two symps. If $\xi\cap\xi'=L$, with $L$ a line, then a point $x\in\xi$ is opposite some point $x'\in\xi'$ if, and only if, $x^\perp\cap L\cap x'^\perp=\varnothing$. In particular, $\xi\cup\xi'$ does not contain any pair of opposite singular $t$-spaces for $t\geq 2$.  
\end{lem}

\begin{lem}\label{3-6space}
Let $\Gamma$ be a Lie incidence geometry of type $\mathsf{E_{7,7}}$ over a field $\K$. Let $U$ and $U'$ be two opposite $3$-spaces and let $W\supseteq U$ and $W'\supseteq U'$ be two $6$-spaces, which are not opposite. Then there exists a plane $\alpha$ in $W$ disjoint from $U$, no point of which is opposite any point of $W'$.  
\end{lem}

We are now ready to determine the various projectivity groups in the exceptional simply laced cases in a geometric way. 

\begin{theorem}\label{caseE}
Let $\Delta$ be a building of type $\mathsf{E_6}$, $\mathsf{E_7}$ or $\mathsf{E_8}$ over the field $\K$. Let $F$ be a simplex of $\Delta$ such that $\Res_\Delta(F)$ is irreducible. Then $\Pi(F)$ and $\Pi^+(F)$ are given as in \emph{\cref{Etabel}}, where the last column contains a checkmark if $\typ(F)$ is polar closed.  Again,  a long hyphen in the table in the column of $\Pi(F)$ means that $\typ(F)$ is not self-opposite and so $\Pi(F)$ is trivially isomorphic to $\Pi^+(F)$ --- it must again be read as a ``bysame'' symbol. Grey rows correspond to projectivity groups which are not necessarily full linear groups.  
\end{theorem}

\begin{table}[tbh]
\renewcommand{\arraystretch}{1.3}
\begin{tabular}{|c|c|c|c|c|c|c|}\hline
Reference&$\typ(\Delta)$& $\Res_\Delta(F)$ &  $\cotyp(F)$ & $\Pi^+(F)$&$\Pi(F)$& \\ \hline
(A1) & & $\mathsf{A_1}$ & & $\PGL_2(\K)$ &$\PGL_2(\K)$&\\ \hline
 \multirow{3}{*}{(A2)} &$\mathsf{E_6}$ & $\mathsf{A_2}$& $\{2,4\}$  & $\PGL_3(\K)$ & $\PGL_3(\K).2$&\\ 
 & $\mathsf{E_6}$ & $\mathsf{A_2}$ & $\neq\{2,4\}$ & $\PGL_3(\K)$ &  ------&\\ 
 & $\mathsf{E_7,E_8}$ & $\mathsf{A_2}$ & & $\PGL_3(\K)$ & $\PGL_3(\K).2$ &\\ \hline
\multirow{3}{*}{(A3)} & $\mathsf{E_6}$ & $\mathsf{A_3}$ & $\{3,4,5\}$ & $\PGL_4(\K)$ & $\PGL_4(\K)$& $\checkmark$\\
&$\mathsf{E_6}$ & $\mathsf{A_3}$ & $\neq\{3,4,5\}$& $\PGL_4(\K)$ & ------&\\ 
&$\mathsf{E_7,E_8}$  & $\mathsf{A_3}$ & & $\PGL_4(\K)$ & $\PGL_4(\K).2$&\\ \hline
\multirow{2}{*}{(A4)}&$\mathsf{E_6}$ & $\mathsf{A_4}$ & & $\PGL_5(\K)$ & ------ &\\ 
&$\mathsf{E_7,E_8}$ & $\mathsf{A_4}$ &  & $\PGL_5(\K)$&$\PGL_5(\K).2$& \\ 
\hline
\multirow{4}{*}{(A5)}&\cellcolor[gray]{0.8}$\mathsf{E_6}$ & \cellcolor[gray]{0.8}$\mathsf{A_5}$ &\cellcolor[gray]{0.8} & \cellcolor[gray]{0.8}$\PSL_6(\K,3)$ & \cellcolor[gray]{0.8}$\PSL_6(\K,3)$ &\cellcolor[gray]{0.8} $\checkmark$ \\ 
&\cellcolor[gray]{0.8}$\mathsf{E_7}$ &\cellcolor[gray]{0.8} $\mathsf{A_5}$ &\cellcolor[gray]{0.8} $\{2,4,5,6,7\}$ &\cellcolor[gray]{0.8} $\PSL_6(\K,2)$ & \cellcolor[gray]{0.8}$\PSL_6(\K,2).2$ &\cellcolor[gray]{0.8} \\ 
&$\mathsf{E_7}$ & $\mathsf{A_5}$ & $2\notin\cotyp(F)$ & $\PGL_6(\K)$ & $\PGL_6(\K).2$ &\\
&$\mathsf{E_8}$ & $\mathsf{A_5}$ &  & $\PGL_6(\K)$ & $\PGL_6(\K).2$ & \\ \hline
(A6)&$\mathsf{E_7,E_8}$ & $\mathsf{A_6}$ &  & $\PGL_7(\K)$ & $\PGL_7(\K).2$ & \\ 
\hline
(A7)&$\mathsf{E_8}$ & $\mathsf{A_7}$ &  & $\PGL_8(\K)$ & $\PGL_8(\K).2$ & \\ \hline
\multirow{2}{*}{(D4)}&$\mathsf{E_6}$ & $\mathsf{D_4}$ &  & $\overline{\mathsf{PGO}}_{8}^\circ(\K)$ & $\overline{\mathsf{PGO}}_{8}(\K)$& \\ 
&$\mathsf{E_7,E_8}$ & $\mathsf{D_4}$ &  & $\overline{\mathsf{PGO}}^\circ_{8}(\K)$ & $\overline{\mathsf{PGO}}^\circ_{8}(\K)$ & $\checkmark$ \\\hline
\multirow{2}{*}{(D5)}&$\mathsf{E_6}$ & $\mathsf{D_5}$ &  & $\overline{\mathsf{PGO}}_{10}^\circ(\K)$ &------ & \\
%&$\mathsf{E_7,E_8}$ & $\mathsf{D_5}$ & $\{1,2,3,4,5\}$ & $\overline{\mathsf{PGO}}_{10}^\circ(\K)$ & $\overline{\mathsf{PGO}}_{10}(\K)$ & \\
&$\mathsf{E_7,E_8}$ & $\mathsf{D_5}$ &   & $\overline{\mathsf{PGO}}_{10}^\circ(\K)$ & $\overline{\mathsf{PGO}}_{10}(\K)$ & \\ \hline
\multirow{2}{*}{(D6)}&\cellcolor[gray]{0.8}$\mathsf{E_7}$ & \cellcolor[gray]{0.8}$\mathsf{D_6}$ & \cellcolor[gray]{0.8} & \cellcolor[gray]{0.8}  $\overline{\mathsf{P\Omega}}_{12}(\K)$&\cellcolor[gray]{0.8} $\overline{\mathsf{P\Omega}}_{12}(\K)$ & \cellcolor[gray]{0.8}$\checkmark$\\  & $\mathsf{E_8}$ & $\mathsf{D_6}$ & &  $\overline{\mathsf{PGO}}^\circ_{12}(\K)$& $\overline{\mathsf{PGO}}^\circ_{12}(\K)$ & $\checkmark$\\ \hline
(D7)&$\mathsf{E_8}$ & $\mathsf{D_7}$ &  & $\overline{\mathsf{PGO}}_{14}^\circ(\K)$ & $\overline{\mathsf{PGO}}_{14}(\K)$ & \\ \hline
(E6)&$\mathsf{E_7,E_8}$ & $\mathsf{E_6}$ & & $\mathsf{PGE_6}(\K)$ &$\mathsf{PGE_6}(\K).2$  & \\  \hline
(E7)&$\mathsf{E_8}$ & $\mathsf{E_7}$ &  & $\PGE_7(\K)$ & $\PGE_7(\K)$ & $\checkmark$\\  \hline
\end{tabular}\vspace{.3cm}
\caption{Projectivity groups in the exceptional cases $\mathsf{E_6,E_7,E_8}$\label{Etabel}}
\end{table}

\begin{proof} In \cref{algapp} we noted that the projectivity groups for $\typ(\Delta)=\mathsf{E_8}$ are the full linear groups. We see no point in reproving this geometrically, the more because the $\mathsf{E_8}$ case is geometrically the most intricate case with the longer arguments; the interested reader can consult the first author's thesis for a detailed version of that. The general projectivity groups in the $\mathsf{E_8}$ case can be deduced from \cref{MR1} and \cref{MR2}, or by analogy with the $\mathsf{E_7}$ cases. The $\mathsf{E_6}$ and $\mathsf{E_7}$ cases do reveal some beautiful geometry and complementary views and we provide the detailed proofs. 
 
The \framebox{case (A1)} was handled in \cref{MR3}. 
We now handle the other cases. Note that we may again assume that $|\K|\geq 3$ as otherwise the linear groups are unique. \medskip

\framebox{Cases (A2) and (A3)} Every subdiagram of type $\mathsf{A_2}$ or $\mathsf{A_3}$ of $\mathsf{E}_r$, $r=6,7,8$, is contained in a subdiagram of type $\mathsf{A_3}$ or $\mathsf{A_4}$, respectively. Then the assertions all follow from \cref{corgate} and \cref{caseA}. \medskip

\framebox{Case (A4)}  If $2\notin\cotyp(F)$ for $\mathsf{E_6}$, or if $\cotyp(F)\neq\{1,2,3,4\}$ for $\mathsf{E_7,E_8}$, then we can again embed the diagram of $\Res_\Delta(F)$ in diagram of type $\mathsf{A_5}$ and use \cref{corgate} and \cref{caseA}. 

Now suppose $2\in\cotyp(F)$ for $\mathsf{E_6}$ and $\cotyp(F)=\{1,2,3,4\}$ for $\mathsf{E_7}$ and $\mathsf{E_8}$.  Then the assertion follows from \cref{corgate} and Case (A$^{**}$) for $r=5$ of \cref{caseD}.

\framebox{Case (A5)} In the Coxeter diagram of type $\mathsf{E_7}$, every subdiagram of type $\mathsf{A_5}$ not containing the node of type $2$ is contained in one of type $\mathsf{A_6}$ and hence the assertion in this case follows from \cref{corgate}.

Now suppose $\Delta$ is the building of type $\mathsf{E_6}$ over the field $\K$, and $F$ is a vertex of type $2$.

We argue in the corresponding Lie incidence geometry of type $\mathsf{E_{6,1}}$. There, $F$ is a $5$-space. Let $F_1,F_2,F_3$ be three $5$-spaces, with both $F_1$ and $F_3$ opposite both $F$ and $F_2$, and with $F$ adjacent to $F_2$, and $F_1$ adjacent to $F_3$, that is, $\pi_0:=F\cap F_2$ and $\pi_1:=F_1\cap F_3$ are planes. We also initially assume that $\pi_0$ and $\pi_1$ are opposite. Consider the projectivity $\rho\colon F\per F_1\per F_2\per F_3\per F$. We claim that $\rho$ fixes each point of $\pi_0$. Indeed, let $p_0\in\pi_0$ be such a point. Then clearly, since $F\cap F_2$ contains $p_0$, the projectivities ${F\per F_1\per F_2}$ and ${F_2\per F_3\per F}$ fix $p_0$, hence $p_0^\rho=p_0$ and the claim is proved.  Likewise, $\rho$ fixes each point of $F$ collinear with a point of $\pi_1$. The set of such points forms a plane $\pi_0'$ of $F$, disjoint from $\pi_0$. Choosing a basis of $F$ in $\pi_0\cup\pi_0'$, a matrix of $\rho$ is a diagonal matrix with diagonal elements three times $1$ and three times some scalar $k\in\K$. We now show that $k$ can be arbitrary. This is equivalent to showing that, 
\begin{itemize}\item[(*)] given $F_1,F_2$ and $F_3$ as above, given a line $L_0$ in $F$ containing points $x_0\in\pi_0$ and $x_0'\in\pi_0'$, and given two points $p,q\in L_0\setminus\{x_0,x_0'\}$, we can re-choose $F_3$ through $\pi_1$ such that $\rho$ maps $p$ to $q$.   \end{itemize}
We now prove (*). Let $p_1$ be the projection of $p$ onto $F_1$ and let $p_2$ be the projection of $p_1$ onto $F_2$. If $p$ and $p_2$ were not collinear, then the symp $\xi(p,p_2)$ would contain $p_1$ and $\pi_0$, leading to additional points in $\pi_0$ collinear to $p_1$ inside $\xi(p,p_2)$, contradicting the fact that $F$ and $F_1$ are opposite and hence $p_1$ is far from $F$. Hence there is some singular $4$-space $U$ containing $\pi_0,L$ and $p_2$. (Note that, since $U$ intersects $F$ in a $3$-space, \cref{44'5}$(ii)$ implies that $U$ is really a $4$-space and not a $4'$-space.) Set $\xi:=\xi(x_0,p_1)$. Then $\xi$ contains $p,q,p_1,p_2$ and the unique point $x_1\in\pi_1$ collinear to $x_0'$. It is clear that $\pi_1$ intersects $\xi$ in only $x_1$, as otherwise there would be a point of $\pi_1$ collinear to $x_0$, contradicting the fact that $\pi_0$ and $\pi_1$ are opposite.   So, \cref{plane-symp} yields a plane $\alpha\subseteq\xi$ collinear to $\pi_1$. \cref{44'5} implies that $\alpha$ and $\pi_1$ are contained in a unique $4'$-space $U_2$, which is itself contained in a unique $5$-space $F_3$. Now both $q$ and $p_2$ are (inside $\xi$) collinear to all points of respective lines of $\alpha$, implying that they are collinear to a common point $p_3\in F_3$. Now (*) follows. 

It now also follows that the set of such projectivities $\rho$ (as above with $\pi_0$ and $\pi_1$ opposite) is geometric (they are the homologies with two disjoint planes as centres, see \cref{exAn}). Now we drop the assumption of $\pi_0$ being opposite $\pi_1$. We claim that in this more general case, the projectivity $\rho$, as defined above, is the product of homologies with disjoint planes as centres. Indeed, set $\pi_0':=\proj^{F_1}_F(\pi_1)$  as above. If $\pi_0'$ is disjoint from $\pi_0$, then by \cref{Tits}, $\pi_0$ and $\pi_1$ are opposite. Now we treat the other cases. Set $d=\dim(\pi_0\cap\pi_0')$ and note that $d=-1$ is precisely the case we already proved. 
\begin{compactenum} \item[$d=0$] Let $\pi_2$ be a plane in $F$ sharing a line with $\pi_0$ but disjoint from $\pi_0'$. Then it is easy to check that the unique $4$-space $U$ containing the $3$-space generated by $\pi_2$ and $\pi_0$ is disjoint from $\proj^{F_1}_{F_2}(\pi_1)$. Hence there exists a $5$-space $F_2'\neq F$ containing $\pi_2$ and opposite both $F_1$ and $F_3$, and we have that $\pi_1$ is opposite both $F\cap F_2'$ and $F_2\cap F_2'$. We can now write $\rho$ as the product of $F\per F_1\per F_2'\per F_3\per F$ and the conjugate of $F_3\per F_2'\per F_1\per F_2\per F_3$ by $F\per F_3$, reducing this case to the case $d=-1$, which we already proved. 
\item[$d=1$] Let $\pi_2$ be a plane in $F$ sharing a line with $\pi_0$ and exactly one point (necessarily in $\pi_0$) with $\pi_0'$. Then, similarly as in the case $d=0$, we can choose a $5$-space $F_2'\neq F$ through $\pi_2$ opposite both $F_1$ and $F_3$ and such that $\pi_1$ has a unique point collinear to some point of $F\cap F_2'$, and that point is also the unique point of $\pi_1$ collinear to some point of $F_2\cap F_2'$. We have hence reduced this case to two times the case $d=0$, which we proved above.
\item[$d=2$] This case is similarly reduced to the case $d=1$. We leave the (straightforward) details to the reader. 
\end{compactenum} The claim is proved. Hence, thanks to \cref{conn}, we can apply \cref{upanddown} and obtain that $\Pi^+(F)$ is generated by all homologies with disjoint planes as centres. This group contains $\PSL_6(\K)$ and then clearly corresponds to all $6\times6$ matrices with a determinant equal to some non-zero 3th power.     Also, $\Pi(F)=\Pi^+(F)$ by virtue of \cref{MR1}. 

Now suppose $\cotyp(F)=\{2,4,5,6,7\}$ in case of $\mathsf{E_7}$.   Here we can take for $F$ a pair consisting of a $5$-space $W$ and a symp $\xi$ containing $W$ in the Lie incidence geometry of type $\mathsf{E_{7,7}}$ over the field $\K$. We employ the same method as in the previous case (Case $\mathsf{A_5}$ in $\mathsf{E_6}$), noting that a projectivity $\{W,\xi\}\per \{W',\xi'\}\per \{W'',\xi''\}$, where the simplices $\{W,\xi\}$ and $\{W'',\xi''\}$ are adjacent, is trivial as soon as $W=W''$, and so we may always assume that in such a (sub)sequence $W\neq W''$ and $\xi=\xi''$. However, since the action of the projectivity is apparently independent of the symps $\xi$ and $\xi'$, we may only consider projections from $5$-spaces onto $5$-spaces. Hence let $W_1,W_2,W_3$ be three $5$-spaces with both $W_1$ and $W_3$ opposite both $W$ and $W_2$, and $\Sigma_1:=W_1\cap W_3$ and $\Sigma_0:=W\cap W_2$ $3$-spaces such that the symps $\xi_{0}$ and $\xi_{1}$ containing $W,W_2$, and $W_1, W_3$, respectively, are also opposite. Similarly as in the previous case (type $\mathsf{A_5}$ inside $\mathsf{E_6}$), we may from the beginning assume that $\Sigma_0$ and $\Sigma_1$ are opposite $3$-spaces. Set $L_0:=\proj^{W_1}_W(\Sigma_1)$. Then, by  \cref{Tits} $L_0$ and $\Sigma_0$ are disjoint. Set $L_2:=\proj^{W_1}_{W_2}(\Sigma_1)$, then likewise $L_2$ and $\Sigma_0$ are disjoint. Let $x_0$ be an arbitrary point on $L_0$. Then inside $\xi_0$ one sees that there is a unique point $x_2$ on $L_2$ collinear to $x_0$. We claim that the projectivity $\rho_1\colon W\per W_1\per W_2$ maps $x_0$ to $x_2$. Indeed, set $W_1':=\proj^{\xi_1}_{\xi_0}(W_1)$.  Then, again by \cref{Tits}, $W_1'$ is disjoint from both $W$ and $W_2$. Set $U_1:=\proj^{W}_{W_1}(x_0)$ and $U_1':=\proj^{\xi_1}_{\xi_0}(U_1)$ and note that $\Sigma_1\subseteq U_1$. Then $U_1'\subseteq W_1'$. Since $x_0$ is at distance~$2$ from each point of $U_1$, it follows by \cref{point-sympE7} that $x_0$ is collinear to all points of $U_1'$. Hence $x_0$ is contained in the unique $5'$-space $V_0$ of $\xi_0$ containing $U_1'$. Likewise, if $x_2'=\proj^{W_1}_{W_2}(U_1)$, then $x_2'\in V_0$. Hence $x_0$ and $x_2'$, which is contained in $L_2$ as $U_1$ contains $\Sigma_1$, are collinear. Consequently, $x_2'=x_2$ and the claim is proved. 

It now also follows that $\rho_3\colon W_2\per W_3\per W$ maps $x_2$ back to $x_0$, since $x_0$ is the unique point on $L_0$ collinear to $x_2$. Consequently, the projectivity $\rho\colon W\per W_1\per W_2\per W_3\per W$ fixes each point of $L_0$. It is easy to see that it also fixes every point of $\Sigma_0$. Hence it is a homology corresponding to a diagonal matrix with the diagonal consisting of four times a $1$ and two times a scalar $k\in \K^\times$. If we can now show that every non-zero scalar $k$ can occur, then, similarly to the case $\mathsf{A_5}$ in $\mathsf{E_6}$, using \cref{upanddown} and \cref{conn}, we are done.  

But it follows from the arguments in the previous paragraphs that the projectivity $\rho_1$ coincides with the projectivity $W\per W_1'\per W_2$ inside $\xi_0$. Likewise the projectivity  $\rho_3$ coincides with the projectivity $W_2\per W_3'\per W$ inside $\xi_0$, with $W_3':=\proj^{\xi_1}_{\xi_0}(W_3)$. Now the assertion follows with exactly the same arguments as Case (A**) in the proof of \cref{caseD}. 

This concludes Case (A5). 

\framebox{Case (A6)} In a Coxeter diagram of type $\mathsf{E_7}$ a subdiagram of type $\mathsf{A}_6$ necessarily has type $\{1,3,4,5,6,7\}$. We work in the Lie incidence geometry $\Delta$ of type $\mathsf{E_{7,7}}$, where $F$ is a 6-space.  %is feasible, we prefer to consider the  Lie incidence geometry of type $\mathsf{E_{7,1}}$, where the proof is entirely similar to the case (A7) of $\mathsf{A_7}$ in $\mathsf{E_8}$, which we, however, leave out. \hvm{What do we do here? Adopt the below proof to this case?}%So we refer to that case for the details. 

Let $F=W_0$ and $W_2$ be two $6$-spaces in $\Delta$ intersecting in a $3$-space that we denote by $U$. Let $W_{1}$ be a $6$-space opposite both $W_0$ and $W_2$. Then $U$ projects to a plane $\alpha$ in $W_{1}$. Projection here means that each point of $U$ is symplectic to each point of $\alpha$. Let $U'$ be a $3$-space in $W_{1}$ that has no intersection with $\alpha$, and note that $U$ is opposite $U'$. Let $W_{3}$ be a $6$-space that intersects $W_{1}$ in $U'$ and is opposite both $W_0$ and $W_3$; this is possible since we assume $|\K|\geq 3$. Set  \[\rho \colon W_{0} \per W_{1} \per W_{2} \per W_{3} \per W_{0}.\] 

We claim that all points of $U$ are fixed under $\rho$. Indeed, a point of $W_{0} \cap W_{2}$ first maps to a hyperplane of $W_{1}$, then back to itself, then to a hyperplane of $W_{3}$ and again back to itself. That means $U$ is fixed pointwise under $\rho$. The claim is proved.

The projection of $U'$ onto $W_{0}$ is a plane that we will denote by $\beta$. Similarly to the proof of the previous claim, we find that $\beta$ is stabilised by $\rho$. We now intend te show that it is pointwise fixed. A point $p$ of $\beta$ maps to a hyperplane $H$ of $W_{1}$ that contains $U'$ and intersects $\alpha$ in a line. The projection of $H$ onto $W_{2}$ is a point that we will denote by $p'$. We claim that  the point $p$ is collinear to $p'$.

Indeed, suppose for a contradiction that $p$ is not collinear to $p'$. Then, since $p$ and $p'$ are both collinear to each point of $U$, they are contained in a unique symp $\xi$. Also, there is a unique symp $\xi_H$ containing the $5'$-space $H$ (by definition of $5'$-spaces). Since $p$ is not opposite any point of $H$, it follows from \cref{point-sympE7} that $p$ is close to $\xi_H$ and hence collinear to some $5'$-space $V_p$ of $\xi_H$. Similarly, $p'$ is collinear to some $5'$-space $V_{p'}$ of $\xi_H$. If $V_p\cap V_{p'}$ were empty, then $p$ and $p'$ would be opposite (as follows from considering an apartment through $V_p$ and $V_{p'}$). If $V_p\cap V_{p'}$ were $3$-dimensional, then $\xi$ and $\xi'$ would intersect in a $5$-space, and no point of $U$ could be opposite any point of $U'$ (by \cref{point-sympE7}), a contradiction. Now suppose $V_p\cap V_{p'}=K$ is a line. Then, clearly, also $\xi\cap\xi_H=K$. But this contradicts \cref{sympline} and the fact that $U\subseteq\xi$ is opposite $U'\subseteq\xi_H$. Hence $V_p=V_{p'}$, implying that $p$ and $p'$ are contained in the same unique $6$-space containing $V_p$, and so $p$ and $p'$ are collinear after all. The claim is proved. 
  
Now we claim that all points of $\beta$ are fixed under $\rho$. Indeed, the point $p$ projects to the hyperplane $H$ in $W_{1}$. This hyperplane projects to $p'$ in $W_{2}$. This already implied that $p$ and $p'$ are collinear. Likewise, $p'$ and the projection $p''$ onto $W_0$ of its projection onto $W3$ are collinear. Since $p''\in\beta$ and the point $p$ is the only point of $\beta$ that $p'$ is collinear to (as $p'^\perp\cap W_0$ is $4$-dimensional), we conclude $p=p''$ and the claim is proved.

Let $xy$ be a line in $W_{0}$ connecting a point $x \in U$ with a point $y \in \beta$. Let $a$ and $b$ be two distinct points on $xy$ not equal to either $x$ or $y$. We claim that we can re-define $W_3$ such that  $\rho$ maps $a$ to $b$. Let $a'$ and $y'$ be the images of $a$ and $y$, respectively, under $W_{0} \per W_{1} \per W_{2}$. Since $y\perp y'$ by one of our previous claims, $\langle a,b,a'\rangle$ is a singular plane and $ba'$ and $yy'$ intersect in a point $s$.

Let $W'$ be a $6$-space through $U$ and $s$. The projection of $W'$ onto $U'$ is a $6$-space that is not opposite $W'$ and that we will denote by $W_{3}$.
Since $W_{3}$ and $W'$ are not opposite, there exists, by \cref{3-6space}, a plane $\gamma_{3}$ in $W_{3}$ such that no point of $\gamma_{3}$ is opposite any point of $W'$. % and there exists a plane $\gamma'$ in $W'$ such that no point of $\gamma'$ is opposite any point of $W_{3}$.

Since both $y$ and $y'$ are not opposite any point of $H$, and hence also not opposite any point of $U'$, the same thing holds for $s$. It follows that $s$ is not opposite any point of $W_3$, as $W'$ and $\gamma_3$ generate $W_3$. This, in turn, implies that $a'$ and $b$ are not opposite the same points of $W_3$ which means, in other words, that $b$ is the image of $a'$ under $W_2\per W_3\per W_0$. The claim is proved. 

By \cref{MR0}, $\Pi^+(W_0)$ contains $\PSL_7(\K)$. By the above, it also contains all diagonal matrices with diagonal $(1,1,1,1,k,k,k)$, $k\in\K^\times$ arbitrarily, and the entries $k$ can be anywhere. This readily implies that $\Pi^+(W_0)$ contains all matrices with determinant a third power, and since $3$ and $7$ are relatively prime, we conclude $\Pi^+(W_0)=\PGL_8(\K)$ and $\Pi(W_0)=\PGL_8(\K).2$. 

\framebox{Case (D5)} We first consider the case of a Coxeter diagram of type $\mathsf{E_6}$. Without loss of generality, we may assume that $F$ has type $6$. Hence we consider $F$ as a symp in a geometry of type $\mathsf{E_{6,1}}$ over the field $\K$.

Let $p_{1}$ be a point in $\Delta$ and $\xi_{0}$ a symp opposite $p_{1}$ in $\Delta$. Let $U$ be a maximal singular subspace in $\xi_{0}$. Then $U$ is a $4$-space. Let $\xi_{2}$ be another symp through $U$ opposite $p_1$. Opposite a $4$-space are lines. Let $L$ be a line through $p_{1}$ opposite $U$ and $V := \proj_{\xi_{0}}(L)$. Let $p_{3}$ be any point on $L$ opposite both $\xi_0$ and $\xi_2$,  so that we have a projectivity $\rho: \xi_{0}\per p_{1}\per \xi_{2}\per p_{3}\per\xi_{0}$. We will show that $\rho$ fixes $U$ and $V$ pointwise.

First let $x$ be a point in $U$. Then $x$ projects to a symp $\xi(x,p_{1})$, then back to $x$, since $x \in \xi_{0} \cap \xi_{2}$, then to a symp $\xi(x,p_{3})$ and then again back to $x$.

Now let $y$ be a point in $V$. The point $y=y_{0}$ projects to a symp $\xi(y_{0},p_{1}) = \xi_{y}$ and then to a point $y_{2} \in \xi_{2}$. Suppose $y_{0}$ and $y_{2}$ were not collinear. The symp $\xi_{y}$ has to contain the closure of $y_{0}$ and $y_{2}$. Both $y_{0}$ and $y_{2}$ are collinear to a $3$-space of $U$. The intersection of these $3$-spaces contains a plane. That means that the closure has to contain a plane of $U$ that then had to be contained in $\xi_{y}$. But  that contradicts the fact that $U$ and $p_{1}$ are opposite, because $p_{1}$ would have to be collinear to elements of that plane. It follows that $y_{0} \perp y_{2}$. Now, since $V = \proj_{\xi_{0}}(L)$, we see that $L\subseteq\xi_y$. So $y_{2}$ continues mapping to $\xi_y$ and then back to $y_{0}$. Hence points of $V$ are fixed.

Next we want to show that we can always define $p_{3}$ on $L$ in a way, such that the projectivity $\rho$ defined above maps an arbitrary point $p$ on a line $xy$, with $x \in U$ and $y \in V$, to another arbitrary point $q$ on $xy$ for $p \notin U,V$ and $q \notin U,V$. Given $U, V, L$ and $p_{1}$ as before and a line $xy$ as described above, let $p$ be an arbitrary point on $xy$ not in $U$ or $V$. Then projecting $p$ to $p_{1}$ yields a symp $\xi(p,p_{1})$ that projects to a point $p_{2}$ onto $\xi_{2}$. Let $y_{2} := \proj_{\xi_{2}} ( \proj_{p_{1}} (y))$. By the previous paragraph, the points $x,y,y_2$ generate a singular plane, which contains $p,q$ and $p_2$. Let $a:=p_{2}q \cap yy_{2}$. Suppose $a$ were collinear to $p_{1}$. Then $a$ would be in $\xi(p,p_{1})$ and $\xi(p,p_{1})$ would contain the plane $\langle x,y,y_{2}\rangle$ and in particular the line $xy$. But that contradicts the fact that $\xi(p,p_{1})$ intersects $\xi_{0}$ only in $p$. It follows that $a$ is not collinear to $p_{1}$. That means $a$ is collinear to a different point of $L$ that we will define as $p_{3}$. This point $p_3$ is not collinear to $p_2$ as otherwise $\xi(p,p_1$ would contain $L$, forcing $p\in V$, a contradiction. Since $a$ and $p_{2}$ are in $\xi(p_{3}, p_{2})$,  $\xi(p_{3}, p_{2})$ contains the whole line $ap_{2}$ and hence the point $q$. With that it follows that $p$ maps to $\xi(p,p_{1})$ to $p_{2}$ to $\xi(p_{3}, p_{2}) = \xi(p_{3}, q)$, and finally to $q$. 

Now \cref{PGO}$(i)$ proves the assertion. 

In a Coxeter diagram of type $\mathsf{E_7}$ (or $\mathsf{E_8}$), a subdiagram of type $\mathsf{D_5}$ is always contained in a subdiagram of type $\mathsf{E_6}$, and so we can apply \cref{corgate}, the previous paragraphs, and \cref{MR1}. 

\framebox{Case (D4)} Each subdiagram of type $\mathsf{D_4}$ in a diagram of type $\mathsf{E}_n$, $n=6,7,8$, is contained in a subdiagram of type $\mathsf{D_5}$. It follows that, if $F$ is a simplex of cotype $\mathsf{D_4}$ in a building $\Delta$ of type $\mathsf{E}_n$, $n=6,7,8$, then there is a subsimplex $F'\subseteq F$ of cotype $\mathsf{D_5}$. By the previous case and \cref{MR0}, the stabiliser of $F'$ in the little projective group $\Aut^\dagger(\Delta)$ of $\Delta$ acts on $\Res_\Delta(F')$ as the full linear (type preserving) group of automorphisms. Hence the stabiliser of $F$ in $\Aut^\dagger(\Delta)$, acting on $\Res_\Delta(F)$ contains the stabiliser in the full linear type preserving group of $\Res_\Delta(F')$ of the vertex $F\setminus F'$. This is clearly also the full linear type preserving group of $\Res_\Delta(F)$.

Now, in case of $\typ(\Delta)=\mathsf{E_6}$, it follows from \cref{MR1} that $\Pi^+(F)$ has index $2$ in $\Pi(F)$, and so $\Pi(F)$ is the full linear group of the corresponding polar space of $\Res_\Delta(F)$.  In case of $\mathsf{E_7}$ or $\mathsf{E_8}$, \cref{MR1} implies that $\Pi(F)=\Pi^+(F)$. 

\framebox{Case (D6)} We treat the case of type $\mathsf{D_6}$ inside type $\mathsf{E_7}$. Let $\xi$ be a symp of the geometry of type $\mathsf{E_{7,7}}$ over the field $\K$. We first claim that $\Pi(\xi)$, which is equal to $\Pi^+(\xi)$ by \cref{MR1}, contains all homologies pointwise fixing two $\xi$-opposite maximal singular $5$-spaces. Let $M_{13}$ and $M$ be two such subspaces of $\xi$. Let $\xi_3$ be an arbitrary symp distinct from $\xi$ and containing $M_{13}$. Let $\xi_2$ be a symp opposite both $\xi$ and $\xi_3$ (and note that this implies that each point of $\xi_2$ is opposite some point of $\xi$). There is a unique maximal singular subspace $M_{24}$ contained in $\xi_2$ each point of which is collinear to some point of $M$, that is, $M_{24}=\proj^{\xi}_{\xi_2}(M)$. Let $L$ be any given line in $\xi$ joining a point $p_{13}\in M_{13}$ and $p\in M$. Choose two points $q,q'\in L\setminus\{p_{13},p\}$. Set $q_2=\proj_{\xi_2}(q)$ and $q_3=\proj_{\xi_3}(q_2)$. 

If $q$ were not collinear to $q_3$, then the symp containing them would contain a $3$-dimensional subspace of $M_{13}$ and $q_2$; this would imply that $q_2$ is close to $\xi$, contradicting \cref{point-sympE7} in view of our remark in the previous paragraph that says that $q_2$ is opposite some point of $\xi$. Hence $\<q,q_3,q'\>$ is a plane $\pi$, contained in the symp $\zeta$ containing $p_{13}$ and $q_2$. Let $\xi'$ be any symp containing, $M_{24}$, but distinct from $\xi_2$. Let $p_{24}$ be the unique point of $\xi_2$ collinear to $p$, and note $p_{24}\in M_{24}$, and that $p_{24}$ and $q_2$ are collinear. Hence $p_{24}\in\zeta$. This implies that $\zeta\cap\xi'$ is either a line or a $5$-space through $p_{24}$. In the latter case $p_{13}$, being collinear with more than one point of that intersection, is close to $\xi'$, contradicting \cref{point-sympE7} and the fact that $M_{24}$ is opposite $M_{13}$, and hence $p_{13}$ is opposite points of $\xi'$. Hence $\zeta\cap\xi'$ is a line $K\ni p_{24}$. If $q_2$ were not collinear to $K$, then $\zeta$ would contain a $3$-space of $M_{24}$, again a similar contradiction (since $\zeta$ contains $p_{13}$). The planes $\pi$ and $\<q_2,K\>$ are easily seen to be opposite in $\zeta$, hence there is a unique point $q_4\in\<q_2,K\>$ collinear to both $q_3$ and $q'$. Now let $\xi_4$ be the symp containing $M_{24}$ and $q_4$, whose existence follows from \cref{symppencilE7}. Then one checks that $\xi_4$ is opposite both $\xi$ and $\xi_3$, and the projectivity $\xi\per\xi_2\per\xi_3\per\xi_4\per\xi$ pointwise fixed both $M_{13}$ and $M$, and maps $q$ to $q'$.  This proves the claim. 

Now, if we want to apply \cref{upanddown}, then we have to show that every projectivity \[\rho\colon\xi_0\per \xi_1\per \xi_2\per \xi_3\per \xi_0,\] with $M_0:=\xi_0\cap \xi_2$ and $M_1=\xi_1\cap\xi_3$ singular $5$-spaces, is the product of similar projectivities, but with $M_0$ opposite $M_1$. So suppose $M_0$ and $M_1$ are not opposite. As for the case of type $\mathsf{A_5}$ in type $\mathsf{E_6}$, there are $3$ cases to consider, and they are again all quite similar to each other, so we consider for example the case where the set of points of $M_0$ collinear to a point of $M_1$ is a line $L$ (the other possibilities are a $3$-space and the whole space $M_1$). Then we consider an appropriate $5$-space $M_2$ in $\xi_0$ intersecting $M_0$ in a $3$-space contained in $M_0$, and disjoint from $L$. Then we find a symp $\xi_2'$ containing $M_2$,  opposite both $\xi_1$ and $\xi_3$, and intersecting $\xi_2$ in a $5$-space opposite $M_1$. As in the case of type $
\mathsf{A_5}$ in type $\mathsf{E_6}$, we can now write $\rho$ as the product of $\xi_0\per \xi_1\per \xi_2'\per \xi_3\per \xi_0$ and the conjugate of $\xi_3\per \xi_2'\per \xi_1\per \xi_2\per \xi_3$ by $\xi_0\per \xi_3$. 

Now we can use \cref{PGO}, \cref{exDn} and, thanks to \cref{conn}, also \cref{upanddown} to conclude that $\Pi^+(F)=\overline{\mathsf{P\Omega}}_{12}(\K)$.

\framebox{Case (E6)} Let $\Gamma$ be the parapolar space of type $\mathsf{E_{7,7}}$ over the field $\K$. Let $p_1,p_2,p_3$ be three mutually opposite points of $\Gamma$. If we show that the self-projectivity $\rho\colon p_1\per p_2\per p_3\per p_1$ is always a symplectic polarity, then  \cref{triangles} and \cref{E77points} implies that $\Pi(p)$ is generated by all the symplectic polarities. By Proposition 6.8$(i)$ of \cite{Sch-Sas-Mal:22}, $\rho$ pointwise fixes a subbuilding of type $\mathsf{F_4}$. More exactly, if $\Res_\Delta(p_1)$ is viewed as a parapolar space $\Gamma_{p_1}$ of type $\mathsf{E_{6,1}}$ with the lines through $p_1$ as points, then $\rho$ pointwise fixes a geometric hyperplane inducing in $\Gamma_{p_1}$ a geometry of  type $\mathsf{F_{4,4}}$ over the field $\K$. It follows from \cite{Sch-Sas-Mal:18} that $\rho$ is a symplectic polarity.  Now \cref{exE6} and \cref{polE6} show that $\Pi^+(p)$ is $\mathsf{PGE_6}(\K)$ and $\Pi(p)$ is $\mathsf{PGE_6}(\K).2$. 
\end{proof}

This concludes the proofs of all our main results. We conclude the paper with some remarks. 

\begin{remark}\label{omega}
It now follows from \cref{caseE} that $\overline{\mathsf{P\Omega}}_{12}(\K)$ does not always coincide with $\overline{\PGO}_{12}(\K)$. Indeed, if it did, then the special projectivity groups in the buildings of type $\mathsf{E_7}$ of all irreducible residues of types contained in $\mathsf{D_6}$ would be the full linear groups. This contradicts the second grey row of \cref{Etabel} for fields containing non-square elements. 
\end{remark}

%\begin{remark}
%The argument for case $\mathsf{E_7}$ in $\mathsf{E_8}$ of the proof of \cref{caseE} could also be used for the cases of  $\mathsf{D_6}$ in $\mathsf{E_7}$ and  $\mathsf{A_5}$ in $\mathsf{E_6}$, if we would use the corresponding long root geometries. We chose to use the simpler and more accessible Lie incidence geometries of types $\mathsf{E_{7,7}}$ and $\mathsf{E_{6,1}}$, respectively, instead, also as a warm-up for the more complicated cases such as $\mathsf{A_5}$ in $\mathsf{E_7}$ and $\mathsf{D_7}$ in $\mathsf{E_8}$. 
%\end{remark}

\begin{remark}
In the course of the proof of \cref{caseE} we do not really need the full strength of Lemmas~\ref{PGO}$(i)$ and \ref{polE6}, since we know by \cref{MR0} that also the little projective group is already contained in the group we want to generate. This knowledge would simplify the proof, since we would only have to prove that the little projective group together with the said homologies generate the full linear group.  
\end{remark}

\begin{remark}\label{nsl}
One could ask what to expect of the case where the diagram is not simply laced. For starters, the description of all spherical buildings is more complicated. Secondly, \cref{MR3} will not hold anymore in full generality. Indeed, there are polar spaces of rank $n$ where $\Pi^+(F)$ is not permutation equivalent to $\PGL_2(\K)$, for $F$ of cotype $n$, even if the set of maximal singular subspaces through a submaximal singular subspace carries in a natural way the structure of a projective line over $\K$ (like a symplectic polar space). However, analogues, appropriately phrased, of Theorems~\ref{MR1} and~\ref{MR2} should still hold. Also, \cref{MR0} remains through across all types. In the split case, the algebraic approach via the Chevalley groups and the weight lattices also still works, if one performs a computation in the symplectic generalised quadrangles analogously to the one we did  in \cref{algapp} in $R_{s,j}$ in case it is a projective plane over a field. For type $\mathsf{F_4}$, all special projectivity groups are the full linear groups as the weight lattice coincides with the root lattice. 
\end{remark}

\textbf{Acknowledgment.} The authors are grateful to Gernot Stroth for an illuminating discussion concerning the structure and action of Levi complements in Chevalley groups and to the referee for some very helpful comments and suggesting different approaches at various points, in particular the approach using algebraic groups in \cref{algapp}.


\begin{thebibliography}{99}

\bibitem{Abr-Bro:08}
{P.~Abramenko \& K.~S.~Brown}, \emph{Buildings. Theory and applications}, Graduate Texts in Math. \textbf{248}, Springer, New York, 2008. 

\bibitem{Bor:69}
{A.~Borel}, \emph{Linear Algebraic Groups}, Graduate Texts in Mathematics, Springer, New York, 1969.

\bibitem{Bou:68} N. Bourbaki, \emph{Groupes et Alg\`ebres de Lie}, Chapitres 4, 5 et 6, \emph{Actu. Sci. Ind.} \textbf{1337}, Hermann, Paris, 1968.

\bibitem{Bue-Coh:13} F. Buekenhout \& A. Cohen, \emph{Diagram Geometry Related to Classical Groups and Buildings}, A Series of Modern Surveys in Mathematics  {\bf 57}, Springer, Heidelberg, 2013.

\bibitem{lines} S. Busch \& H. Van Maldeghem, A characterisation of lines in finite Lie incidence geometries of classical type, \emph{Discrete Math}. \textbf{349} (2026), \#114711, 15pp.

\bibitem{lines2} S. Busch \& H. Van Maldeghem, Lines and opposition in finite Lie incidence geometries of exceptional type, in progress.

\bibitem{Car:72} R. W. Carter, \emph{Simple Groups of Lie Type,}  John Wiley \& Sons, London, New York, Sydney, Toronto, 1972, \emph{Classic Series}, 1989.

\bibitem{Car:93} R. W. Carter, \emph{Finite Groups of Lie Type, Conjugacy Classes and Complex Characters},  John Wiley \& Sons, London, New York, Sydney, Toronto, 1985, \emph{Classic Series},1993. 

\bibitem{Che:54} C. Chevalley, \emph{The Algebraic
Theory of Spinors}, Columbia University Press, New York, 1954.

%\bibitem{Coh-Iva:07} A. M. Cohen \& G. Ivanyos, Root shadow spaces, \emph{European J. Combin.} \textbf{28} (2007), 1419--1441.

\bibitem{Atlas} J. H. Conway, R. T. Curtis, S. P. Norton, R. A. Parker, R. A. Wilson, \emph{Atlas of Finite
Groups}. Clarendon Press, Oxford, 1985.

\bibitem{Coo:76} B. N. Cooperstein, {Some geometries associated with parabolic representations of groups of Lie type}, \emph{Canad. J. Math.} \textbf{28} (1976), 1021--1031. 

\bibitem{Coo:77} B. N. Cooperstein, A characterization of some Lie incidence structures, \emph{Geom. Dedicata} \textbf{6} (1977), 205--258.

\bibitem{CurKanSei}
{C.~Curtis, W.~Kantor, G.~Seitz}, The $2$-Transitive Permutation Representations of the Finite Chevalley Groups, \emph{Trans. of the Amer. Math. Soc.} \textbf{218}  (1976), 1--59.

\bibitem{Sch-Sas-Mal:18}  A. De Schepper, N. S. N. Sastry \& H. Van Maldeghem, Split buildings of type $\mathsf{F_4}$ in buildings of type $\mathsf{E_6}$, \emph{Abh. Math. Sem. Univ. Hamburg} \textbf{88} (2018), 97--160.

\bibitem{Sch-Sas-Mal:22}  A. De Schepper, N. S. N. Sastry \& H. Van Maldeghem, Buildings of exceptional type in buildings of type $\mathsf{E_7}$, \emph{Dissertationes Math.} \textbf{573} (2022), 1--80.

\bibitem{ADS-JS-HVM} A. De Schepper, J. Schillewaert \& H. Van Maldeghem, On the generating rank and embedding rank of the Lie incidence geometries, {\em Combinatorica} {\bf 44} (2024), 355--392.

\bibitem{derderij} A. De Schepper, J. Schillewaert, H. Van Maldeghem \& M. Victoor, Construction and characterisation of the varieties of the third row of the Freudenthal--Tits magic square, \emph{Geom. Dedicata} \textbf{218} (2024), Paper No. 20, 57pp. 

\bibitem{ADS-HVM} A. De Schepper and H. Van Maldeghem, On inclusions of exceptional long root geometries of type $E$, {\em Innov. Inc. Geom}, {\bf 20} (2023), 247--293.

\bibitem{Dieu:43} J. Dieudonn\'e, Les d\'eterminants sur un corps non commutatif, \emph{Bull. Soc. Math. France} \textbf{71} (1943), 27--45.

\bibitem{Dieu:62} J. Dieudonn\'e, \emph{La G\'eometrie des Groupes Classiques}, 2nd ed., Springer, Berlin, 1963.

\bibitem{GoLySo}
{D.~Gorenstein, R.~Lyons, R.~Solomon}, \emph{The classification of the finite simple groups}, Amer. Math. Soc. Surveys and Monographs \textbf{40(3)}, 1994.

\bibitem{Kas-Shu:02} A. Kassikova \& E. Shult, {Point-line characterisations of Lie incidence geometries}, \emph{Adv. Geom.} \textbf{2} (2002), 147--188.

\bibitem{Kna:88} N. Knarr,  Projectivities of generalized polygons,
{\it Ars Combin.}\ \textbf{25B} (1988), 265--275.

\bibitem{Muh-Ped-Wei:15} B. M\"uhlherr, H. Petersson \& R. M. Weiss, \emph{Descent in Buildings}, Annals of Mathematics Studies \textbf{190}, Princeton University Press, 2015. 

\bibitem{Muh-Ron:95}  B. M\"uhlherr \& M. A. Ronan, Local to global structure in twin buildings, \emph{Invent. Math.} \textbf{122} (1995), 71--81.

\bibitem{FGQ} S. E. Payne \& J. A. Thas, \emph{Finite Generalized Quadrangles}, Research notes in Math. \textbf{110}, Pittman, 1984; second edition: Europ. Math. Soc. Series of Lectures in Mathematics, 2009. 


\bibitem{Shu:11} E. E. Shult, \textit{Points and Lines: Characterizing the Classical Geometries}, Universitext, Springer-Verlag, Berlin Heidelberg, 2011.
\bibitem{Ste:20} Y. Stepanov, \emph{On an Octonionic Construction of the Groups of Type $\mathsf{E_6}$ and $\mathsf{^2E_6}$}, PhD thesis, Queen Mary University of London, 2020. 

\bibitem{Ste:68} R. Steinberg, \emph{Lectures on Chevalley Groups}, Yale University, Notes prepared by John Faulkner and Robert Wilson, June 1968. 
\texttt{https://www.math.utah.edu/~ptrapa/math-library/steinberg/steinberg-yale-notes.pdf}

\bibitem{Tem-Tha-Mal:11} B. Temmermans, J. A. Thas \& H. Van Maldeghem, Domesticity in projective spaces, \emph{Innov. Incid. Geom.} \textbf{12} (2011), 141--149. 

\bibitem{Tim:01}  F.G. Timmesfeld, \emph{Abstract root subgroups and simple groups of Lie type}, Monographs
in Mathematics \textbf{95}, Birkh\"auser, 2001.

\bibitem{Tits:57} J. Tits, Sur la g\'eometrie des $R$-espaces, \emph{J. Math. Pure Appl.} (9) \textbf{36} (1957), 17--38.

\bibitem{Tits:64}
{J. ~Tits}, \emph{Algebraic and abstract simple groups}, Ann. of Math \textbf{80} (1964), 313--329.

\bibitem{Tits:74} J. Tits, \emph{Buildings of spherical type and finite BN-pairs}, Lecture Notes in Math. \textbf{386}, Springer-Verlag, Berlin, 1974 (2nd printing, 1986).

\bibitem{Tits:77} J. Tits, Endliche Spiegelungsgruppen, die als Weylgruppen
auftreten, \emph{Invent. Math.} \textbf{43} (1977), 283--295.

\bibitem{Tits-Wei:02} J. Tits \& R. Weiss, \emph{Moufang Polygons}, Springer Monographs in Mathematics, Springer, 2002.

\bibitem{Mal:98} H. Van Maldeghem, \emph{Generalized Polgons}, Monographs in Mathematics \textbf{93}, Birkhaeuser, 1998.

\bibitem{Mal:12} H. Van Maldeghem, Symplectic polarities in buildings of type $\mathsf{E_6}$, \emph{Des. Codes Cryptogr.} \textbf{65} (2012), 115--125. 

\bibitem{Mal-Vic:19} H. Van Maldeghem \& M. Victoor,  Combinatorial and geometric constructions of spherical buildings, \emph{Surveys in Combinatorics} 2019, Cambridge University Press (ed. A. Lo et al.), \emph{London Math. Soc. Lect. Notes Ser.} \textbf{456} (2019), 237--265.

\bibitem{Mal-Vic:22} H. Van Maldeghem \& M. Victoor, On Severi varieties as intersections of a minimum number of quadrics, \emph{Cubo} \textbf{24} (2022), 307--331.

\end{thebibliography}
\end{document}